\documentclass[a4paper]{amsart}
\usepackage[utf8]{inputenc}

\usepackage{amssymb}

\usepackage{amsthm,mathtools}
\usepackage[active]{srcltx}
\usepackage{color}
\usepackage{enumerate}
\usepackage[all]{xy}
\usepackage{enumitem}
\usepackage{float}
\usepackage[all]{xy}
\usepackage{float}
\usepackage{adjustbox}
\usepackage[normalem]{ulem}
\usepackage{multirow}

\renewcommand{\k}{\Bbbk}
\newcommand{\ku}{\Bbbk}

\newcommand{\nc}{\newcommand}

\newcommand{\ot}{\otimes}

\newcommand{\ydh}{{}^{H}_{H}\mathcal{YD}}

\newcommand{\fk}{\mathcal{FK}}

\newcommand{\rg}{\rangle}
\renewcommand{\lg}{\langle}

\newcommand{\ra}{\rightharpoonup}

\newcommand{\dcopro}{\operatorname{dcopro}}
\renewcommand{\d}{\operatorname{d}}
\newcommand{\e}{\operatorname{e}}
\newcommand{\car}{\operatorname{char}}
\newcommand{\Hom}{\operatorname{Hom}}

\newcommand{\ex}{\operatorname{Exp}}
\newcommand{\id}{\operatorname{id}}

\newcommand{\sgn}{\operatorname{sign}}

\renewcommand{\sl}{\mathfrak{sl}}

\newcommand{\Cleft}{\operatorname{Cleft}}

\renewcommand{\mod}{\operatorname{-mod}}

\newcommand{\BB}{\mathbb{B}}
\newcommand{\N}{\mathbb{N}}

\newcommand{\s}{\mathbb{S}}
\newcommand{\G}{\mathbb{G}}

\newcommand{\I}{\mathbb{I}}

\newcommand{\mE}{\mathcal{E}}
\renewcommand{\P}{\mathrm{P}}

\newcommand{\Bq}{\mathfrak{B}}
\renewcommand{\O}{\mathcal{O}}

\newcommand{\cc}{{\bf c}}

\newcommand{\ul}{\underline{\lambda}}

\newcommand{\gap}{\texttt{GAP}}

\newcommand{\eps}{\epsilon}
\newcommand{\veps}{\varepsilon}

\def\pf{\begin{proof}}
\def\epf{\end{proof}}

\def\bs{\boldsymbol}

\newcommand{\Z}{\operatorname{Z}}
\newcommand{\B}{\operatorname{B}}
\newcommand{\C}{\operatorname{C}}

\newcommand{\Dchaintwo}[3]{\xymatrix@C-4pt{\overset{#1}{\underset{\ }{\circ}}\ar
		@{-}[r]^{#2}
		& \overset{#3}{\underset{\ }{\circ}}}}

\nc{\ben}{\begin{enumerate}[(i)]}
\nc{\een}{\end{enumerate}}

\def\trid{\vartriangleright}

\numberwithin{equation}{section}

\theoremstyle{plain}
\newtheorem{theorem}{Theorem}[section]
\newtheorem{lemma}[theorem]{Lemma}

\newtheorem{proposition}[theorem]{Proposition}
\newtheorem{corollary}[theorem]{Corollary}

\newtheorem{question}[theorem]{Question}
\newtheorem*{theoremone}{Theorem A}
\newtheorem*{theoremtwo}{Theorem B}

\theoremstyle{remark}
\newtheorem{remark}[theorem]{Remark}

\newtheorem*{acknowledgement*}{Acknowledgement}

\allowdisplaybreaks

\title[Hopf cocycles over $\s_3$]{Hopf cocycles associated to pointed and copointed deformations over $\s_3$}

\author{Agust\'in Garc\'ia Iglesias}
\address{A.G.I, J.I.S: FaMAF-CIEM (CONICET),
	Medina Allende S/N,
	Universidad Nacional de C\'ordoba,
	Ciudad Universitaria, C\'ordoba (X5000HUA),
	República Argentina. }
\email{agustingarcia@unc.edu.ar, jose.ignacio.sanchez@mi.unc.edu.ar}

\author{José Ignacio Sánchez}

\begin{document}

\begin{abstract}
The algebras in the classification of finite dimensional pointed and copointed Hopf algebras over the symmetric group on three letters are all Hopf cocycle deformations of their associated graded objects. We show that these cocycles are generically pure, namely they are not co-homologous to exponentials of Hochschild 2-cocycles; apart from very specific examples.

As well, we provide a full description of these cocycles on a given basis.
\end{abstract}

\maketitle

\section{Introduction}\label{sec:intro}

We follow the ideas developed and collected in \cite{GS} to write down explicitly the Hopf 2-cocycles involved in the classification results of pointed and copointed Hopf algebras over the symmetric group $\s_3$. We show that these cocycles cannot be generically obtained as exponentials of Hochschild 2-cocycles, which justifies this explicit computation for a better understanding of these objects. We refer the reader to {\it loc.cit.}~for a thorough introduction to this program.  As well, 

Finite-dimensional pointed and copointed Hopf algebras over $\s_3$ arise as deformations of the Fomin-Kirillov algebra $\fk_3$ over $\k\s_3$ and $\k^{\s_3}$, respectively. Recall that $\fk_3$ denotes the quadratic algebra generated by variables $x_0,x_1,x_2$ and relations
\begin{align}\label{eqn:fk3-intro}
	x_0^2=x_1^2=x_2^2&=0, & 
	x_0x_1+x_1x_2+x_2x_0=x_1x_0+x_2x_1+x_0x_2&=0.
\end{align}

Set $V=\k\{x_0,x_1,x_2\}$, then this is indeed the Nichols algebra associated to the braided vector space $(V,c)$ where the braiding $c\colon V\ot V\to V\ot V$ is determined by:
\begin{align}\label{eqn:braiding-intro}
	c(x_i\ot x_j)=-x_{i\trid j}\ot x_i, \quad i\trid j\coloneqq 2i-j \, (3);
\end{align}
or equivalently to the dual braiding $c'$ see \eqref{eqn:bvs2}. 
As such, it admits realizations on $\ydh$, both for $H=\k\s_3$ and $H=\k^{\s_3}$.
In the pointed case we  consider the more general case of Hopf algebras with coradical $H=\k\G_{3,\ell}$, for
\begin{equation}\label{eqn:g3}
	\G_{3,\ell}=C_3\rtimes C_{2\ell}=\lg s,t | s^3=t^{2\ell}=1, ts=s^2t\rg,
\end{equation}
for which $\s_3\simeq \G_{3,1}$. Here $C_n$ denotes the cyclic group of order $n\in\N$.

This is a 12-dimensional graded algebra; an homogeneous basis is given by:
\begin{align}\begin{split}\label{eqn:basis}
		\mathbb{B}=\{1,x_0,x_1,x_2,x_0x_1,& x_0x_2,x_1x_0,x_1x_2,\\ & x_0x_1x_0,x_0x_1x_2,x_1x_0x_2,x_0x_1x_0x_2\}.
	\end{split}
\end{align}
We refer to \S\ref{sec:basis} for some remarks on the choice of the basis $\BB$.

Following the strategy settled in \cite{AAGMV}, pointed, respectively copointed, Hopf algebras with coradical $H$, are classified in terms of Hopf cocycle deformations $A_\sigma$ of $A\coloneqq\fk_3\# H$. However, the cocycles involved are not known, as this method presents the algebras $A_\sigma$ by generators and relations, see \S\ref{sec:method}. More precisely they arise as left Hopf algebras associated to certain families of cleft objects. 

A key feature of this classification, related to the present work, is that the presentation of the algebras $A_\sigma$ can be very obscure and intricate; hence the explicit knowledge of the cocycles $\sigma$ can provide an alternative to find these presentations, as the multiplication $m_\sigma\colon A_\sigma\ot A_\sigma\to A_\sigma$ is the convolution conjugation $m_\sigma=\sigma\ast m\ast \sigma^{-1}$ of the multiplication $m\colon A\ot A\to A$ of $A$ by $\sigma$, namely:
\begin{align}\label{eqn:mult-def}
	m_\sigma(a,b)=\sigma(a_{(1)},b_{(2)})m(a_{(2)},b_{(2)})\sigma^{-1}(a_{(3)},b_{(3)}), \qquad a,b\in A.
\end{align}

\subsection{Pure cocycles}
There is an alternative way to look for Hopf cocycles for a given Hopf algebra $A$, in terms of Hochschild 2-cocycles with trivial coefficients $\eta\in \Z^2(A,\k)$: Under certain hypotheses, the exponential
$e^\eta=\sum\limits_{k\geq 0}\frac{1}{k!}\eta^{\ast\,k}\in Z^2(A)$.

We say that a Hopf 2-cocycle $\sigma$ is called {\it pure} if it is not cohomologous to an exponential of a Hochschild 2-cocycle, {\it cf.}~\cite[Definition 5.2]{GS}.
Our results show that this is indeed the generic situation, as very restricting conditions are necessary for a Hopf cocycle to be an exponential in the examples.

\subsection{Results}

We refer to \S\ref{sec:pointed-def} and \S\ref{sec:copointed-def} for the description of the Hopf algebras belonging to each classification, namely the pointed Hopf algebras $A_{\lambda,\mu}$, $\lambda,\mu\in\k$, and the copointed Hopf algebras $A_{[\bs\alpha]}$, $\bs\alpha=(\alpha_1,\alpha_2)\in\k^2$. We set $\ul\coloneqq \dfrac{\lambda}{3}$.

We present here a first approach to our results characterizing the cocycles involved in the deformations. 
We recall from our {\it decomposition formula} in \cite[Lemma 3.2]{GS}, see also \eqref{eqn:dec-form2}, that a Hopf cocycle can be reconstructed from the values $\sigma(x_i,b)$, $b\in \mathbb{B}$, $i\in \{0,1,2\}$. We exhibit here these fundamental values and single out the pure cocycles. We refer to Theorems \ref{cociclodiagonal} and \ref{cocicloCOdiagonal} to find the complete tables describing the values for each cocycle on the basis $\BB\times \BB$. These theorems also present the cohomology orbit of such cocycles.

\begin{theoremone}\label{teo:intro1}
	Let $A = \fk_3\#\k\G_{3,\ell}$ and let $A_{\lambda,\mu}$, $\lambda,\mu\in\k$, be a non-trivial finite-dimensional pointed Hopf algebra over $\G_{3,\ell}$. Then it follows that $A_{\lambda,\mu}\simeq A_{\sigma}$, where $\sigma=\sigma_{\lambda,\mu}\#\eps\in Z^2(A)$ and $\sigma_{\lambda,\mu}\in Z^2(\fk_3)^{\G_{3,\ell}}$ 
	is fully determined by the values
	\begin{align*}
		\sigma_{\lambda,\mu}(x_0,x_0)&=\mu, & \sigma_{\lambda,\mu}(x_0,x_1)&=\ul, \\ 
		\sigma_{\lambda,\mu}(x_0,x_0x_1x_0)&=\ul(\ul-\mu), & \sigma_{\lambda,\mu}(x_0,x_1x_0x_2)&=\mu^2+\ul^2,
	\end{align*}
	and the fact that $\sigma_{\lambda,\mu}(x_0,x)=0$ if $x\in\BB$ has even degree.
	
	As well, $\sigma_{\lambda,\mu}$ is pure if and only if $\mu\neq \ul$. In fact, $\sigma_{\lambda,\ul}=e^{\eta_\lambda}$, where
	\begin{equation}\label{eqn:eta-lambda}
		\eta_\lambda=\ul(\xi_0+\xi_2)+2\ul^2\xi_{3}
	\end{equation}
	and $\xi_0,\xi_2,\xi_{3}\in\Z^2(\fk_3,\k)^{\G_{3,\ell}}$ as in \S\ref{subsec-hoch-pointed}.
	\qed
\end{theoremone}

When $\ell=1$, we recover in the following corollary the statement about $\s_3$ from \cite[Theorem 4.10]{GaM}. 

\begin{corollary}\label{cor:s3-intro}
	Let $A=\fk_3\#\k{\s_3}$ and let $L$ be a finite-dimensional pointed Hopf algebra over $\s_3$, with $L\not\simeq \k\s_3$. Then there is $\lambda\in\k$ such that  $L\simeq A_{\sigma_\lambda}$, where $\sigma_\lambda\coloneqq\sigma_{\lambda,\lambda/3}\#\eps\in Z^2(A)$ and $\sigma_{\lambda,\lambda/3}=e^{\eta_\lambda}$, for $\eta_\lambda\in\Z^2(\fk_3,\k)^{\s_3}$ as in \eqref{eqn:eta-lambda}.
\end{corollary}
Thus all deformations over $\s_3$ can be attained using exponentials. This is no longer true when $\ell>1$, nor in the copointed case, see Corollaries \ref{cor:pointed-exponential} and  \ref{cor:copointed-exponential}.

In \cite{GaM}, they consider a Hochschild 2-cocycle $\eta\in\Z^2(\fk_3,\k)$ concentrated in degree one, that is $\eta(x_i,x_j)=\lambda/3$ for $i,j\in\{0,1,2\}$ and zero elsewhere in the basis $\mathbb{B}$. In our language, this is $\eta=\frac{\lambda}{3}(\xi_0+\xi_2)$. In Proposition \ref{pro:fk3-converse}, we see that $e^\eta=\alpha\ra e^{\eta_\lambda}$, for some $\alpha\in\fk_3^\ast$ and $\eta_\lambda$ as in \eqref{eqn:eta-lambda}.

\begin{remark}\label{rem:comparacion}
	Both cocycles $e^\eta$ from \cite{GaM} and $\sigma_{\lambda,\lambda/3}=e^{\eta_\lambda}$ from Corollary \ref{cor:s3-intro} produce, for each  $i\in\{0,1,2\}$, the relation $x_i^2=\frac{\lambda}{3}(1-s^2)=\frac{\lambda}{3}(1-1)=0$, which is not normalized. In this setting, the normalized ($\mu=0$) relation is obtained via the cocycle $\sigma_{\lambda,0}$. This cocycle, however, is not an exponential, nor it is cohomologous to $e^\eta$ or $e^{\eta_\lambda}$. Indeed, the corresponding cleft objects are not isomorphic, even when the Hopf algebra deformations are so.
\end{remark}

Now we turn to the setting of copointed deformations. We refer to \eqref{eqn:etabi} for the definition of the Hochschild cocycles $\xi_i^i\in \Z^2(\fk_3,\k)$, $i=0,1,2$.
\begin{theoremtwo}\label{teo:intro2}
	Consider $A = \fk_3\#\k^{\s_3}$ and let $A_{[\bs\alpha]}$, $\bs\alpha=(\alpha_1,\alpha_2)\in\k^2$, be a non-trivial finite-dimensional copointed Hopf algebra over $\s_3$. Then $A_{[\bs\alpha]}\simeq A_{\sigma}$, where $\sigma=\sigma_{\cc}\#\eps\in Z^2(A)$
	is associated to the triple $\cc=(c_0,c_1,c_2)$ so that 
	\begin{align}\label{eqn:alphas}
		\alpha_1=c_0-c_2, \qquad \alpha_2=c_0-c_1
	\end{align}
	and $\sigma_{\cc}\in Z^2(\fk_3)^{\k^{\s_3}}$ is completely determined by the values
	\begin{align*}
		\sigma_{\cc}(x_i,x_j)&=\delta_{i,j}c_i, & \sigma_{\cc}(x_1,x_0x_1x_2)&=c_2\alpha_1, & \sigma_{\cc}(x_2,x_0x_1x_0)&=-c_1\alpha_1
	\end{align*}
	and the fact that $\sigma_{\cc}(x_i,x)=0$, $i\in X$, when $x\in\BB$ has even degree.

	The cocycle $\sigma_{\cc}$ is pure if and only if at least two of the parameters $\{c_0,c_1,c_2\}$ are different from zero. In fact,
	$\sigma_{(c,0,0)}=e^{c\xi_0^0}$, $\sigma_{(0,c,0)}=e^{c\xi_1^1}$  and $\sigma_{(0,0,c)}=e^{c\xi_2^2}$.
	\qed
\end{theoremtwo}

The following analysis was motivated by a remark pointed out to us by Cristian Vay, whom we thank for a interested reading of our work.

\begin{remark}
	A subtle difference between the braided and linear cocycles is worth mentioning here. While any triple $\cc=(c_0,c_1,c_2)\in\k^3$ can be chosen to define the braided cocycle $\sigma_{\cc}$, or equivalently the braided cleft object $\mE_{\cc}$ see \cite[Lemma 6.5, Proposition 7.2]{GIV2}, when dealing with the associated linear cocycle $\sigma=\sigma_{\cc}\#\eps$, it is enough to consider triples $\cc=(c_0,c_1,c_2)\in\k^3$ so that $c_0+c_1+c_2=0$. 
	Indeed in (the proof of) \cite[Proposition 29]{AV} the authors show that the corresponding (linear) cleft objects are nonzero under this assumption. Which is enough, as \[\cc=(c_0,c_1,c_2)\mapsto \cc'=(c_0-\kappa/3,c_1-\kappa/3,c_2-\kappa/3), \quad \kappa\coloneqq c_0+c_1+c_2,\] leads to an isomorphism $A_{[\bs\alpha]}\simeq A_{[\bs\alpha']}$, for  $\bs\alpha=\bs\alpha(\bs\cc)$ and $\bs\alpha'=\bs\alpha'(\bs\cc')$ as in \eqref{eqn:alphas}. 
	
	In terms of Hopf algebras, we show in Corollary \ref{cor:copointed-exponential} that $A_{[\bs\alpha]}$ is isomorphic to a nontrivial exponential cocycle deformation of $\fk_3\#\k^{\s_3}$ only if $[\bs\alpha]=[1:0]$.
\end{remark}

\subsection{On the proofs}

For the proof, we follow the steps introduced in \cite{GS}: 
\begin{enumerate}[label=(\alph*)]
	\item[(a)] Compute the $H$-linear section $\gamma_{\bs\lambda}\colon \fk_3\to \mE_{\bs\lambda}$ in the basis $\mathbb{B}$.
	\item[(b)] Write down the table of values of $\sigma_{\bs\lambda}$ in terms of $\gamma_{\bs\lambda}$ as in \eqref{eqn:cociclo-gamma-intro}.
	\item[(c)] Describe the space of $H$-invariant Hochschild 2-cocycles $\Z^2(\fk_3,\k)^H$; and find necessary and sufficient conditions so that $e^\eta$ is a Hopf cocycle.
	\item[(d)] Check which cocycles $\sigma_{\bs\lambda}$ are cohomologous to exponentials $e^\eta$: that is we describe the orbit $\alpha\ra\sigma_{\bs\lambda}$ and compare the tables.
\end{enumerate}


We proceed in parallel lines: on the one hand, we perform all of these calculations by hand, and present here the most prototypical arguments. We reserve longer calculations for \cite{Stesis}. As well, we use \cite{GAP} with \cite{GBNP} to produce the values in (b), as well as certain computations for (c) and (d). We refer to code and log files in \texttt{https://www.famaf.unc.edu.ar/\~{}aigarcia/}. This presents a possible alternative for results on Fomin-Kirillov algebras $\fk_n$, $n\geq 4$, cf.~\S\ref{subsec:fomin}. We compute the section by hand, but we remark that there is a related algorithm in \cite[\S4.5]{AG}, which could be adapted to deal with item (a).

\subsubsection{\gap\, computations}\label{sec:gap}

In the file(s) \texttt{cocycle-(co)pointed.g} we compute for each pair of elements $a, b\in\BB$, the values $\sigma(a,b)\in\k$ and $\alpha_d\ra\sigma(a,b)\in\k$.

First, for each $a, b\in\BB$ we compute $\d_a\coloneqq a_{(1)}\ot 1\ot a_{(2)}$ and $\e_b\coloneqq 1\ot b_{(1)}\ot b_{(2)}$. 
This is done as follows: in the algebra $\fk_3\ot \fk_3\ot \fk_3$ (here $\ot$ stands for a braided tensor product), we define:
\begin{align*}
	d_i&\coloneqq x_i\ot 1\ot 1+1\ot 1\ot x_i, &  e_j&\coloneqq 1 \ot x_j\ot 1+1 \ot 1\ot x_j, \qquad i,j\in X.
\end{align*}
For any $a=p_a(x_0,x_1,x_2)$ and $b=p_b(x_0,x_1,x_2)\in \BB$ non-commutative monomials on $\k\lg x_0,x_1,x_2\rg$ we let $\d_a=p_a(d_0,d_1,d_2)$, $\e_b=p_b(e_0,e_1,e_2)$. Next we braid and multiply to obtain the {\it double coproduct}:  
\[
\dcopro(a,b)\coloneqq a_{(1)}\ot b_{(1)}\ot a_{(2)}b_{(2)}, \qquad a,b\in\BB.
\]

Hence we compute for each pair of elements $a, b\in\BB$, the values $\sigma(a,b)\in\k$ via:
\begin{align}\label{eqn:gap-cocycle}
	a\ot b&\longrightarrow \dcopro(a,b)\longrightarrow \gamma(a_{(1)})\gamma(b_{(1)})\gamma^{-1}(a_{(2)} b_{(2)}).
\end{align}
We remark that the final multiplication is in the cleft deformation $\mE$.
The input are the values $\gamma(x)$, $x\in \BB$. As an output we get the nonzero values $\sigma(a,b)$, $a,b\in\BB$. 

Next, we compute in this file the semi-orbit $\alpha\ra\sigma$, $\alpha\in U(\fk_3^\ast)$, namely:
\begin{align}\label{eqn:gap-orbit}
	a\ot b&\longrightarrow \dcopro(a,b)\longrightarrow \sigma(a_{(1)},b_{(1)})\alpha_d^{-1}(a_{(2)} b_{(2)}),
\end{align}
following the description of the dual units $U(\fk_3^\ast)$ in \S\ref{sec:dualunits} and the observation that follows Lemma \ref{lem:alpha} therein. 
The input are the values $\alpha_d(x)$, $x\in \BB$, $d\in\k$, and the final output are the nonzero values of the orbit $\alpha_d\ra\sigma(a,b)$, $a,b\in\BB$. 

In turn, in \texttt{eta2-(co)pointed.g} we calculate in an analogous fashion the coproducts $\d_a'=a_{(1)}\ot 1\ot a_{(2)}\ot 1$ and $\e_b'=1\ot b_{(1)}\ot 1\ot b_{(2)}$ in $\fk_3^{\ot 4}$. Next, we braid and multiply to  compute $\eta^{\ast2}(a,b)\in \k$ as
\begin{align}\label{eqn:gap-eta}
	a\ot b&\rightarrow \dcopro'(a,b)\coloneqq a_{(1)}\ot b_{(1)}\ot a_{(2)}\ot b_{(2)}\rightarrow \eta(a_{(1)}, b_{(1)})\eta( a_{(2)}, b_{(2)}).
\end{align}
Once again, the output is the (unique) nonzero value $\eta^{\ast2}(x_0x_1x_0x_2,x_0x_1x_0x_2)$; the input being the (generic) values of $\eta(x,y)$, $x,y\in\BB$.

\subsection{Organization}

The paper is organized as follows. In Section \ref{sec:prels} we write down some previous concepts on Hopf algebras and deformations. 
We present our results for pointed Hopf algebras in Section \ref{sec:pointed} and the corresponding results for copointed Hopf algebras in Section \ref{sec:copointed}. 

\section{Preliminaries}\label{sec:prels}

We work over a field $\k$ algebraically closed with $\car \k=0$.
We set,  for each $\theta\in\N$,  $I_\theta\coloneqq\{1,\dots,\theta\}\subset \N$; we write $\mathbb{P}^\theta$ for the $\theta$-dimensional $\k$-projective space.

Let us fix $H$ to be a Hopf algebra with bijective antipode $S$.
We denote by
$G(H)$ the set of group-like elements on $H$ and by $\P(H)$ the primitive elements.
Let $H_0$ denote the coradical of $H$. Then $H$ is pointed when $H_0=\k G(H)$ and copointed when $H_0=\k^G$ for a (finite) non-abelian group $G$. We write $\delta_t$, $t\in G$, for the canonical idempotents in $\k^G$.
We  use Sweedler's notation for the comultiplication $\Delta$ of $H$. We write $\ydh$ for the (braided) category of Yetter-Drinfeld modules over $H$. Recall that if $V,W\in \ydh$ then these are simultaneously $H$-modules and comodules, up to certain compatibilities.

A (normalized) Hopf 2-cocycle  is a
linear map $\sigma: H\ot  H \to \ku$
which is convolution invertible and  such that $\sigma(x, 1) = \sigma(1, x) = \eps(x)$ and 
\begin{align*}
	\sigma(x_{(1)}, y_{(1)}) \sigma(x_{(2)} y_{(2)}, z) &=
	\sigma(y_{(1)}, z_{(1)}) \sigma(x, y_{(2)}z_{(2)}), \quad x,y,z\in H.
\end{align*}
We write $Z^2(H)$ for the set of normalized Hopf 2-cocycles in $H$. 

The group of convolution units $U(H^\ast)\subset H^\ast$ acts on $Z^2(H)$ as follows: for each $\alpha\in U(H^\ast)$ and $\sigma\in Z^2(H)$, then 
$\alpha\rightharpoonup\sigma\in Z^2(H)$, where
\begin{equation}\label{eqn:alpha-action-sigma}
	(\alpha\rightharpoonup\sigma)(x,y)=\alpha(x_{(1)})\alpha(y_{(1)})\sigma(x_{(2)},y_{(2)})\alpha^{-1}(x_{(3)}y_{(3)}), \qquad x,y\in H.
\end{equation}
This defines an equivalence  $\sigma\sim \sigma'$, $\sigma,\sigma'\in Z ^2(H)$, if and only if there is $\alpha$ such that $\sigma=\alpha\rightharpoonup\sigma'$. In this case, we say that $\sigma$ and $\sigma'$ are cohomologous Hopf 2-cocycles; we write $\sigma\sim_\alpha\sigma'$ to highlight the map $\alpha\in U(H^\ast)$.

Given $\sigma\in Z^2(H)$, it is possible to define a new multiplication $m_\sigma$ as in \eqref{eqn:mult-def} on the vector space $H$ in such a way that this produces a new Hopf algebra $H_\sigma$, with the same comultiplication $\Delta$. We say that a Hopf algebra is a cocycle deformation of $H$ if it is isomorphic to $H_\sigma$, for some $\sigma$.

A (right) cleft object for $H$ is a right comodule algebra $C$ for which there exists a convolution invertible comodule isomorphism $\gamma\colon H\to C$, with $\gamma(1)=1$, known as the {\it section}. In this setting, it is possible to define a Hopf 2-cocycle for $H$ via
\begin{align}\label{eqn:cociclo-gamma}
	\sigma(x,y)=\gamma(x_{(1)})\gamma(y_{(1)})\gamma^{-1}(x_{(2)}y_{(2)}), \quad x,y\in H.
\end{align}

Moreover, there is well-known one-to-one correspondence between the (isomorphism classes of) cleft objects $\Cleft(H)$ and the (cohomology classes of) Hopf 2-cocycles $H^2(H)\coloneqq Z^2(H)/U(H^\ast)$; see \cite{Sch} for details. This equivalence does not carry on to cocycle deformations, but two cohomologous Hopf cocycles do yield isomorphic Hopf algebras.

Recall that a Nichols algebra is a connected graded Hopf algebra $B=\oplus_{n\geq 0} B_n$ in $\ydh$ generated by $V\coloneqq B_1$ and such that $\P(B)=V$. As a braided bialgebra, $B$ is completely determined by the braiding $c_{V,V}$ on the vector space $V$; we write $B=\Bq(V)$. Conversely, a realization of a braided vector space $(V,c)$ -and consequently of $\Bq(V)$- in $\ydh$ is a structure of Yetter-Drinfeld module on $V$ so that $c$ coincides (as a linear map) with the braiding $c_{V,V}$ in the category.
A realization is called {\it principal} when the coaction is diagonal on a basis of $V$, see \cite[Definition 3.5]{GIV1}, also \cite{AS} for further reference on Nichols algebras.

\subsection{Exponential cocycles}\label{sec:expo}

A possible source of Hopf 2-cocycles is a certain subset of Hochschild 2-cocycles $\eta\in\Z^2(H,\k)$. Assume that $H=\oplus_{n\geq 0} H_n$ is graded and $\eta_{|H\ot H_0+H_0\ot H}=\eps$. Then the exponential $e^\eta=\sum_{k\geq 0}\frac{1}{k!}\eta^{\ast\,k}$ defines a map $H\ot H\to \k$ and a sufficient (though not necessary) condition for this map to be a Hopf 2-cocycle is given by the commutations
\begin{align}\label{eqn:conm1}
	[\eta(\id\ot m), \eps\ot \eta]_{\ast}&=0, & [\eta(m\ot \id), \eta\ot \eps]_{\ast}&=0
\end{align}
in the convolution algebra $\Hom(H^{\ot3},\k)$.
We stress that these conditions are actually not necessary, see Remarks \ref{rem:inclusion-strict-pointed} and \ref{rem:exp-copointed}, where we show that in the examples under consideration in this manuscript the set $\C$ of Hochschild cocycles satisfying \eqref{eqn:conm1} is strictly contained in the set of Hochschild cocycles $\overline{\C}$ for which the exponential defines a Hopf 2-cocycle. This was already evident in \cite{GS}. We write $\ex (\C),\ex(\overline{\C})\subset Z^2(H)$ for the corresponding subsets of Hopf cocycles.

Nevertheless, we find, as in \cite{GS}, that cocycles in $\ex(\overline{\C})$ are cohomologous, as Hopf cocycles, to cocycles in $\ex(\C)$. See Remarks \ref{rem:question-p} and \ref{rem:question-c}. Hence the following question thus arises naturally in this context.

\begin{question}\label{q}
	Let $\eta\in \Z^2(H,\k)$ be such that $e^\eta\in Z^2(H)$. Is there a cocycle $\xi\in \Z^2(H,\k)$ satisfying \eqref{eqn:conm1} so that $e^\eta\sim e^{\xi}$ in $Z^2(H)$? 
\end{question}

\subsection{From deformations to cocycles}\label{sec:method} If $H$ is any semisimple Hopf algebra $H$ and $B\in\ydh$ is a (finitely presented) Nichols algebra, the strategy in \cite{AAGMV}  produces Hopf algebras $(A_{\bs\lambda})_{\bs\lambda\in\bs\Lambda}$ out of a suitable defined collection $(\mE_{\bs\lambda})_{\bs\lambda\in\bs\Lambda}$ of cleft objects for $B$ inside $H\mod$. The Hopf algebras $A_{\bs\lambda}$ are isomorphic to the left Schauenburg algebras $L(A,C_{\bs\lambda})$, where $C_{\bs\lambda}$ is the $A$-cleft object given by the bosonization $C_{\bs\lambda}=\mE_{\bs\lambda}\# H$. This implies that $A_{\bs\lambda}$ is a cocycle deformation of $A$; the explicit form of the cocycles is not provided.

On the other hand, the factorizations $C_{\bs\lambda}\leftrightarrow\mE_{\bs\lambda}\# H$, $\gamma\leftrightarrow\gamma_{\bs\lambda}\#\id_H$, of the cleft objects and the corresponding section $\gamma\colon A\to C_{\bs\lambda}$, implies that the associated cocycles $\sigma\in Z^2(A)$, see \eqref{eqn:cociclo-gamma}, actually belong to the subspace $Z_0^2(A)$ of cocycles $\sigma$ such that $\sigma_{H\ot H}=\eps$ and thus factorize, in turn as
\begin{align}\label{eqn:factor}
	Z_0^2(A) \ni \sigma=\sigma_{\bs\lambda}\# \eps, \qquad \sigma_{\bs\lambda}\in Z^2(B)^H,
\end{align}
where $Z^2(B)^H$ denotes the space of $H$-linear cocycles for the braided bialgebra $B$. 

In particular, the cocycles involved in the classification results can be fully described by writing down the set of values $\sigma_{\bs\lambda}(b,b')$ for $b,b'$ in a basis of $B$. In turn, these values can be computed using the section $\gamma_{\bs\lambda}$ as
\begin{align}\label{eqn:cociclo-gamma-intro}
	\sigma_{\bs\lambda}(b,b')=\gamma_{\bs\lambda}(b_{(1)})\gamma_{\bs\lambda}(b'_{(1)})\gamma_{\bs\lambda}^{-1}(b_{(2)}b'_{(2)}), \quad b,b'\in B.
\end{align}

The factorization \eqref{eqn:factor} also carries on to the setting of Hochschild 2-cocycles, by a combination of \c{S}tefan's results in \cite{St}, most remarkably that $\Z^2(A,\k)\simeq \Z^2(B,\k)^H$ and the fact that an exponential Hopf cocycle $e^\eta$ belongs to the set $Z_0^2(A)$ if and only if $\eta$ is $H$-linear, \cite[Proposition 5.5]{GS}. As a result, we can fully restrict our computations to the braided setting. 
The $H$-action on  $\Z^2(B,\k)$ is explicitly described in \cite[Theorem 4.4]{GS}.
\begin{remark}
	In \cite{Gu} a different approach is considered to calculate Hopf cocycles for the quantum groups $\mathrm{U}_q(\sl_2)$, 
	$1\neq q\in\k^{\times}$ and $\mathfrak{u}_q(\sl_2)$, $q$ a primitive root of 1. No connection to exponentials is mentioned there.
\end{remark}

\subsection{The Fomin-Kirillov algebras}\label{subsec:fomin}

The algebra $\fk_3$ as defined by relations \eqref{eqn:fk3-intro} was first defined in \cite{FK}, and independently in \cite{MS}. It belongs to a series of quadratic algebras $\fk_n$, associated with the symmetric groups $\s_n$, $n\geq 3$.

Next, we review some characteristics of $\fk_3$, which are relevant for this work: we justify our choice of basis, describe the possible realizations and the corresponding deformations. We remark that most of these features extend to the case $n>3$.  We refer to \cite{V} for details. 

The braiding $c\colon V\ot V\to V\ot V$ in \eqref{eqn:braiding-intro} is associated to the {\it rack} $X=\{0,1,2\}$, with operation $i\trid j=2i-j\,(3)$. As well, it can be identified with the conjugacy class of transpositions $\O_2^3=\{(12),(13),(23)\}\subset \s_3$, with $\sigma\trid \tau=\sigma\tau\sigma^{-1}$, via \begin{align*}¨
	0 & \leftrightarrow g_0=(12), &
	1 & \leftrightarrow g_1=(23), &
	2 & \leftrightarrow g_2=(13).
\end{align*}

More generally, we identify $X=\{0,1,2\}$ with the subset of the group $\G_{3,\ell}$ given by the conjugacy class $\{t^{2k+1},st^{2k+1},s^2t^{2k+1}\}$ of $t^{2k+1}$, via
\begin{align}\label{eqn:gi}
	i  \leftrightarrow  g_i\coloneqq s^it^{2k+1}, \qquad i=0,1,2.
\end{align}
This defines an action $\G_{3,\ell}\times X\to X$ via the {\it conjugation}:
\begin{align}\label{eqn:action}
	g\cdot i\coloneqq j \quad \text{ if } \quad  g_j=gg_ig^{-1}.
\end{align}

\subsubsection{The basis}\label{sec:basis} As presented in Gra\~na's \textit{Zoo of Nichols algebras} \cite{zoo}, the basis $\BB$ in \eqref{eqn:basis} is obtained  by multiplying one element in each set from left to right in all the possible ways:
\begin{align*}
	\{1, x_0\} && \{1, x_1, x_1x_0\} && \{1, x_2\}.
\end{align*}

This is also the basis produced by \gap\, under the function \texttt{BaseQA}. As well, this is essentially the unique monomial (homogeneous) basis as relations imply
the following identities in $\fk_3$, for $i,j,k \in \{0,1,2\}$ pairwise different:
\begin{align}\label{RelA}
	x_i x_j x_i=x_j x_i x_j=-x_i x_k x_j=-x_j x_k x_i, 
\end{align}
As well, any monomial of degree 4 is $\pm x_0x_1x_0x_2$.

We shall write $\BB^+$ for $\BB\cap\fk_3^+$, namely $\BB^+\coloneqq \BB\setminus\{1\}$.

\subsubsection{Principal realizations}
Next we review some particular realizations of the algebra $\fk_3$ as a Hopf algebra in the Yetter-Drinfeld category $\ydh$, both for $H=\k\G_{3,\ell}$ and $H=\k^{\s_3}$. 
In turn, these give rise to families of Hopf algebras, as deformations of  $\fk_3\# H$.

\subsection{Pointed deformations}\label{sec:pointed-def}
Set $H=\k\s_3$ and recall the elements $g_i$ in \eqref{eqn:gi} and the action \eqref{eqn:action}. Then $(V,c)$ admits a principal YD-realization in $\ydh$, via
\begin{align}\label{eqn:real-s3}
	\omega\cdot x_i&=\sgn(\omega) x_{\omega\cdot i}, & x_i&\mapsto g_i\ot x_{i}, \qquad i\in X, \ \omega\in\s_3.
\end{align}
In this way, $\fk_3$ becomes the only (non-trivial) finite-dimensional Nichols algebra in $\ydh$ cf.~\cite{AHS}.
This point of view allows us to consider a more general setting as  in \cite{GIV1}, in which $\s_3$ becomes a particular case of the family of groups $\G_{3,\ell}=C_3\rtimes C_{2\ell}$, $\ell\in\N$, as in \eqref{eqn:g3}. 
The realization \eqref{eqn:real-s3} is then part of a family of principal YD-realizations, one for each $0\leq k<\ell$, coming from the choice of $g_i=s^it^{2k+1}$ in \eqref{eqn:gi}. For $i\in X$ and $g=s^jt^r\in\G_{3,\ell}$, the realization is as follows:
\begin{align}\label{eqn:real-k}
	g\cdot x_i&=(-1)^r x_{g\cdot i}, & x_i&\mapsto g_i\ot x_{i}.
\end{align}
Thus finite-dimensional pointed Hopf algebras over $\G_{3,\ell}$, coming from principal realization of $\eqref{eqn:braiding-intro}$, are classified in terms of pairs $(\lambda,\mu)\in\k^2$. 
For each $0\leq k<\ell$, these are the Hopf algebras $A_{\lambda,\mu}^{(k)}$ described as the quotients of $T(V)\#\k\G_{3,\ell}$, where $V=\k\{a_0,a_1,a_2\}$, modulo the relations
\begin{align}
	\notag a_0^2=a_1^2=a_2^2&=\mu(1-t^{2(2k+1)}),\\
	\label{eqn:relations-pointed}a_1a_0+a_2a_1+a_0a_2&=\lambda(1-st^{2(2k+1)}),\\
	\notag a_0a_1+a_1a_2+a_2a_0&=\lambda(1-s^2t^{2(2k+1)}).
\end{align}
Every $A_{\lambda,\mu}=A_{\lambda,\mu}^{(k)}$ is a cocycle deformation of $\fk_3\#\k\G_{3,\ell}$, see \cite{GIV1}. The isomorphism classes are in correspondence with classes $[\mu:\lambda]\in\mathbb{P}^1$. 

We remark that when $\ell=1$, namely the  $\s_3$-case, then $t^2=1$ and thus we can normalize $\mu=0$. We obtain a family of algebras $A^\lambda\coloneqq A_{\lambda,0}$, $\lambda\in\k$, given by relations  $a_0^2=a_1^2=a_2^2=0$ and
\begin{align}\label{eqn:relation-pointed-s3}
	a_1a_0+a_2a_1+a_0a_2&=\lambda(1-(132)),  & a_0a_1+a_1a_2+a_2a_0&=\lambda(1-(123)).
\end{align}
In this way, only two  Hopf algebras are obtained up to isomorphism: the graded algebra $A^0=\fk_3\#\k\s_3$ and a single non-trivial deformation $A^1$, as $A^\lambda\simeq A^1, \lambda\neq 0$. 

\subsection{Copointed deformations}\label{sec:copointed-def}
The dual braiding $c'\colon V\ot V\to V\ot V$ is
\begin{align}\label{eqn:bvs2}
	c'(x_i\ot x_j)=-x_{j}\ot x_{j\trid i}.
\end{align}
We let $\s_3$ act on $X$ as above: now $(V,c')$ can be realized in $\ydh$, $H=\k^{\s_3}$,  via 
\begin{align}\label{eqn:coreal}
	\delta_\omega\cdot x_i&=\delta_{\omega,g_i} x_{i}, & \tau&\mapsto \sum_{\omega\in\s_3} \sgn(\omega)\delta_\omega\ot x_{\omega^{-1}\cdot i}, \qquad i\in X, \ \omega\in\s_3.
\end{align}

It follows that any non-trivial finite-dimensional copointed Hopf algebra over $\k^{\s_3}$ is a deformation of $\fk_3\# H$, see \cite{AV}. 
More precisely, finite-dimensional copointed Hopf algebras over $\k^{\s_3}$ are classified in \cite[Theorem 3.5]{AV1} by pairs $\bs\alpha=(\alpha_1,\alpha_2)\in\k^2$ (corresponding to $(-a_1,-a_2)$ in the notation in {\it loc.cit.}). These are the Hopf algebras $A_{\bs\alpha}$ obtained as the quotient of $T(V)\#\k^{\s_3}$ with relations
\begin{align}
	\begin{split}\label{eqn:relations-copointed}
		&a_0^2=\alpha_1(\delta_{(23)}+\delta_{(123)}) + \alpha_2(\delta_{(13)}+\delta_{(132)}), \\
		&a_1^2=-\alpha_2(\delta_{(13)}+\delta_{(123)})+(\alpha_1 - \alpha_2)(\delta_{(12)}+\delta_{(132)}), \\
		&a_2^2=(\alpha_2 - \alpha_1)(\delta_{(12)}+\delta_{(123)}) - \alpha_1(\delta_{(23)}+\delta_{(132)}),\\
		&a_0a_1+a_1a_2+a_2a_0=a_1a_0+a_2a_1+a_0a_2=0.
	\end{split}
\end{align}

Here $\delta_t$, $t\in\s_3$, denote the standard idempotents in $\k^{\s_3}$.
Again, every $A_{\bs\alpha}$ is a cocycle deformation of $\fk_3\#\k^{\s_{3}}$. The associated cleft objects $\mE_{\bs\alpha}$ actually depend on triplets $\cc=(c_0,c_1,c_2)\in\k^3$, in such a way that we recover the parameters in the deformed Hopf algebras as in \eqref{eqn:alphas}.

This is related by the procedure to determine the deformations introduced in \cite{AAGMV}. See also \cite[\S 5.1]{GIV2} for a explicit description in the quadratic case. In short, one defines the deformations $A_{(\alpha_1,\alpha_2)}=A_{(c_0-c_2,c_0-c_1)}$ in terms of the cleft objects $\mE_{\cc}=\mE_{(c_0,c_1,c_2)}$. We write our cocycles in this language.

The isomorphism classes of the algebras $A_{\bs\alpha}$ are given by the orbits of the standard (right) projective $\s_3$-action on $[\bs\alpha]=[\alpha_1:\alpha_2]\in\mathbb{P}^1$, namely
\begin{align}\label{eqn:alphas-action}
	[\alpha_1:\alpha_2]\triangleleft (12)&=[\alpha_2:\alpha_1], & [\alpha_1:\alpha_2]\triangleleft (23)&=[-\alpha_1:\alpha_2-\alpha_1].
\end{align}
We write $A_{[\bs\alpha]}$ to refer to a choice of a representative of the isoclass of $A_{\bs\alpha}$.

\subsubsection{Hochschild 2-cocycles}\label{subsec:hoch2}
Assume that $B=\bigoplus_{n\geq 0} B_n$ is a (non-trivial) graded bialgebra, with $B_0=\k$. We set $\underline{\Delta}(1)=1\ot 1$ and for $b\in B^+\coloneqq\bigoplus_{n\geq 1} B_n$:
\begin{align}\label{eqn:underline}
	\underline{\Delta}(b)\coloneqq\Delta(b)-b\ot 1-1\ot b, \qquad b_{\underline{{(1)}}}\ot b_{\underline{{(2)}}}\coloneqq \underline{\Delta}(b).
\end{align}

We recall from \cite[Lemma 4.1]{GS} that a linear map $\eta:B\ot B\to \k$  is a Hochschild 2-cocycle in $\Z^2(B,\k)$ if and only if the following pair of conditions holds:
\begin{align}\label{eqn:conds-hoch}
	\eta(a,1)&=\eta(1,c)=0, \quad a,c\in B^+, & \text{ and }&&
	\eta(a,bc)&=\eta(ab,c), \quad b\in B.
\end{align}

Thus, if $\BB$ is a basis of $B$, then the set $\{\eta(b,x_i)|b\in\mathbb{B},i\in\I_\theta\}\cup\{\eta(1,1)\}$ determine the values of $\eta$. 
In particular, let $\{x_1,\dots,x_\theta\}$ be a basis of $B_1$, and assume $\{x_1,\dots,x_\theta\}\subset \mathbb{B}$. For each $i,j\in\I_\theta$, the linear map
\begin{align}\label{eqn:etabi}
	\xi_{j}^{i}\colon B\otimes B\to \k, \qquad b\ot b'\mapsto \delta_{b',x_j}\delta_{b,x_i}, \quad b,b'\in\mathbb{B}.
\end{align}
defines a Hochschild 2-cocycle in $B$. Also, we set $\veps\coloneqq \eps_B\ot\eps_B \in\Z^2(B,\k)$.

\subsubsection{Dual units}\label{sec:dualunits}
Following \eqref{eqn:alpha-action-sigma}, we see that to understand the cohomology class $[\sigma]$ of a Hopf cocycle $\sigma\in Z^2(\fk_3)^H$, we need to characterize the group of $H$-linear units $U_H(\fk_3^\ast)$, both for $H=\k\G_{3,\ell}$ and $\k^{\s_{3}}$. 


\begin{lemma}\label{lem:alpha}
	Let $\alpha\in U_H(\fk_3^\ast)$, with $\alpha(1)=1$. Then there is $d\in\k$ so that $\alpha=\alpha_d$ is the map determined by
	\begin{align}\label{eqn:alpha-d}
		\alpha(x_i)=\alpha_d(x_ix_j)=\alpha_d(x_ix_jx_i)&=0, \ i\neq j\in X, & \alpha_d(x_0&x_1x_0x_2)=d.
	\end{align}
\end{lemma}
\pf
Follows from the identification $\fk_3^\ast\simeq \fk_3$ in $\ydh$.
\epf

In particular, if $a,b\in\BB^+$ are not both of degree 4, then \eqref{eqn:alpha-action-sigma} becomes
\begin{align}\label{eqn:presigma}
	(\alpha_d\ra\sigma)(a,b)=
	\sigma(a_{(1)},b_{(1)})\alpha^{-1}(a_{(2)}b_{(2)}), 
\end{align}
while for $a=b=x_0x_1x_0x_2$, we get: 
\begin{multline*}
	(\alpha_d\ra\sigma)(a,b)=\alpha_{d}(a)\alpha_d(b)+\alpha_{d}(a)\alpha_d^{-1}(b)+\alpha_{d}(b)\alpha_d^{-1}(a)\\
	+\sigma(a_{(1)},b_{(1)})\alpha^{-1}(a_{(2)}b_{(2)})=\sigma(a_{(1)},b_{(1)})\alpha^{-1}(a_{(2)}b_{(2)})-d^2.
\end{multline*}
Thus we can restrict to compute \eqref{eqn:presigma} to find the orbit of $\sigma$, as explained in \S\ref{sec:gap}.

\section{Deformations over $\s_3$}\label{sec:pointed}

In this section we compute the cocycles $\sigma\in Z^2(A)$ associated to the pointed deformations of $A=\fk_3\#\k\mathbb{G}_{3,\ell}$ as in \eqref{eqn:relations-pointed}; we particularly recall the realization \eqref{eqn:real-k}. We fix $H=\k\G_{3,\ell}$ and $X=\{0,1,2\}$ We set, for $\lambda\in\k$, the notation $\ul=\frac{\lambda}{3}$. 

As well, we shall make great use of identities \eqref{RelA}.

We start by recalling the cleft objects and by computing the corresponding section in Lemma \ref{lem:gamma-pointed}. Then we use this to write down the orbit of the cocycles in Theorem \ref{cociclodiagonal}. We compute the space of invariant Hochschild 2-cocycles $\Z^2(\fk_3,\k)^{\mathbb{G}_{3,\ell}}$ in Proposition \ref{lemprop:hochschild-pointed}, find when an exponential $e^\eta$, $\eta\in \Z^2(\fk_3,\k)^{\mathbb{G}_{3,\ell}}$ is a Hopf 2-cocycle in Proposition \ref{pro:fk3-converse} and decide which Hopf 2-cocycles $\sigma$ are pure and which ones are cohomologous to some $e^\eta$ in Theorem \ref{thm:fk3-sigma-pure}.

\subsection{Cleft objects}

Recall from \cite{GIV1} that the (braided) cleft objects for $\fk_3$ are the algebras $\mE_{\lambda,\mu}$, $\mu,\lambda\in\k$, generated by $y_0,y_1,y_2$ and relations:
\begin{align*}
	y_0^2=y_1^2=y_2^2&=\mu, &
	y_1y_0+y_0y_2+y_2y_1=y_0y_1+y_1y_2+y_2y_0&=\lambda.
\end{align*}

Thus, $\mE_{\lambda,\mu}$ inherits a basis $\mathbb{B}$ as in \eqref{eqn:basis}. We shall make use of the following relations, which hold for pairwise different $i,j,k\in X$:
\begin{align}\label{relH2}
	y_iy_ky_j&=y_jy_ky_i, & y_iy_jy_i&=y_jy_iy_j- \lambda y_j+\lambda y_i=-y_iy_ky_j-\mu y_k+\lambda y_i.
\end{align}
The structure of (right) $\fk_3$-comodule algebra  $\rho:\mE_{\lambda,\mu}\to \mE_{\lambda,\mu}\ot \fk_3$ is given by
\begin{align}\label{eqn:coaction}
	\rho(y_i)=y_i\ot 1+1\ot x_i; \qquad i\in X.
\end{align}

To determine the associated Hopf 2-cocycle $\sigma_{\lambda,\mu}$, we need an explicit description of the section $\gamma_{\lambda,\mu}:\fk_3\to \mE_{\lambda,\mu}$,
which is the content of the next lemma. We remark that this is an algorithmical computation, which basically comes up to reducing the expression $\rho-(\gamma\ot \id)\Delta$ on each basic element, {\it cf.}~\cite[\S4.5]{AG}. 

As well, we stress that this map is $\G_{3,\ell}$-linear: recall the action induced by the realization \eqref{eqn:real-k} of the braiding \eqref{eqn:braiding-intro} and the definition \eqref{eqn:gi} of $g_i\in\G_{3,\ell}$, $i\in X$.

\begin{lemma}\label{lem:gamma-pointed}
	Let	$\mE_{\lambda,\mu}$ be as above and let $i,j,k\in X$ be pairwise different.
	\begin{enumerate}[leftmargin=*,label=(\alph*)]
		\item  The section $\gamma\coloneqq\gamma_{\lambda,\mu}\colon \fk_3\to \mE_{\lambda,\mu}$ is given by:
		\begin{itemize}
			\item $\gamma(x_i)=y_{i}$, 
			\item $\gamma(x_{i}x_{j})= y_{i}y_{j}-\ul$, 
			\item $\gamma(x_0x_{1}x_0)=y_0y_1y_0+\ul y_1-2\ul y_0$,
			\item $\gamma(x_0x_{1}x_2)=y_0y_{1}y_{2}+\mu y_1-\ul(y_0+y_2)$,
			\item $\gamma(x_{1}x_0x_{2})= y_{1}y_0y_{2}+\mu y_0-\ul(y_1+y_2)$,
			\item $\gamma(x_0x_1x_0x_2)=y_0y_1y_0y_2+\ul(y_1y_2-y_1y_0-y_0y_1)-2\ul y_0y_2$.	
		\end{itemize}
		\item The convolution inverse $\gamma^{-1}:\fk_3\to \mE_{\lambda,\mu}$ is determined by:
		\begin{align*}
			\gamma^{-1}(x_ix_j)&=-y_{k}y_i+\ul, & \gamma^{-1}(x_0x_1x_0x_2)&=\gamma(x_0x_0x_1x_2)+\mu^2+3\ul^2,
		\end{align*}
		and $\gamma^{-1}(b)=-\gamma(b)$, for $b\in\BB$ with positive odd degree.
	\end{enumerate}
\end{lemma}

\begin{proof}
	(a)	The identities $\gamma(x_i)=y_i$, $i\in X$ are clear. To continue, 
	we compute the coproduct for each other $b\in \mathbb{B}$. 
	Then we define $\gamma(b)$ in such a way that:
	the identity $\rho\circ\gamma=(\gamma\ot \id)\circ\Delta$ is verified
	and the map $\gamma$ is $H$-linear.
	
	Following \eqref{RelA}, we compute $\Delta(x_ix_jx_i)$ to deal with an element in degree 3.
	
	Now, let $i,j,k\in X$ be all different, then:
	\begin{align}
		\notag\underline{\Delta}(x_ix_j)=&x_i\ot x_j-x_k\ot x_i.
		\\
		\label{eqn:coprod1}\underline{\Delta}(x_ix_jx_i)=&
		x_ix_j\ot x_i-x_ix_k\ot x_j+x_i\ot x_jx_i+x_kx_i\ot x_i+x_j\ot x_ix_j.
		\\
		\notag \underline{\Delta}(x_0x_1x_0x_2)=& x_0x_1x_0\ot x_2-x_0x_1x_2\ot x_1-x_1x_0x_2\ot x_0+x_0\ot x_1x_0x_2\\
		\notag&+x_1\ot x_0x_1x_2-x_2\ot x_0x_1x_0-x_0x_2\ot x_1x_2+x_0x_1\ot x_1x_0\\
		\notag&
		+x_1x_0\ot x_0x_1
		-x_1x_2\ot x_0x_2.
	\end{align}
	
	In the case of the basic elements $x_ix_j$, we necessarily have that 
	\[\gamma(x_ix_j)=y_iy_j-c_{ij}, \qquad c_{ij}\in\k.
	\] Then, the relations in $\mE_{\lambda,\mu}$ and the action of $H$ forces $c_{ij}=\ul$, $i,j\in X$. 
	
	To find the other values of $\gamma$, we note that for $i,j,k\in X$, all different:
	\begin{align*}
		\rho(y_iy_jy_i)-(\gamma\ot \id)\Delta(x_ix_jx_i)=&(y_iy_jy_j-\gamma(x_ix_jx_i))\ot 1\\
		&+\ul(1\ot x_i-1\ot x_j+1\ot x_i).
	\end{align*}
	
	Hence $\gamma(x_ix_jx_i)=y_iy_jy_i+\ul y_j-2\ul y_i\stackrel{\eqref{relH2}}{=}-y_iy_ky_j-\mu y_k+\ul y_i+\ul y_j$.
	
	Now, if $k\in X=\{0,1,2\}$, then
	\begin{align*}
		\gamma(g_k\cdot(x_ix_jx_i))=-y_{k\trid i}y_{k\trid j}y_{k\trid i}-
		\ul y_{k\trid j}+2\ul y_{k\trid i}=g_k\cdot\gamma(x_ix_jx_i).
	\end{align*}
	This shows that $\gamma$ is $\G_{3,\ell}$-linear for this basic elements. Finally, by \eqref{RelA}, we get that $\gamma(x_1x_0x_2)=y_1y_0y_2+\mu y_0-\ul(y_2+y_1)$ and $\gamma(x_0x_1x_2)=y_0y_1y_2+\mu y_1-\ul(y_0+y_2)$.
	
	Finally, we proceed in the same manner for $\gamma(x_0x_1x_0x_2)$. We get
	\begin{align*}
		(\gamma\ot \id)\Delta(&x_0x_1x_0x_2)-\rho(y_0y_1y_0y_2)
		=(y_2y_0y_1y_0-\gamma(x_2x_0x_1x_0))\ot 1\\
		&-\ul\ot x_0x_1+\ul(y_1+y_2)\ot x_0+(\ul y_2-2\ul  y_0)\ot x_1-\ul\ot x_1x_0\\
		&+(\ul y_1-2\ul y_0)\ot x_2+\ul\ot x_1x_2-2\ul\ot x_0x_2+\ul\ot x_2x_0.
	\end{align*}
	Now, we also get that
	\begin{multline*}
		(\gamma\ot \id)\Delta(x_0x_1x_0x_2)-\rho(y_0y_1y_0y_2-2\ul y_0y_2+\ul (y_1y_2- y_1y_0- y_0y_1))\\
		=(\gamma(x_0x_1x_0x_2)-(y_0y_1y_0y_2-2\ul y_0y_2+\ul (y_1y_2- y_1y_0- y_0y_1)))\ot 1.
	\end{multline*}
	Hence, we define;
	\begin{align*}
		\gamma(x_0x_1x_0x_2)=y_0y_1y_0y_2+\ul (y_1y_2- y_1y_0- y_0y_1)-2\ul y_0y_2;
	\end{align*}
	it follows that this is indeed $\G_{3,\ell}$-invariant.
	
	(b)  It is easy to see that $\gamma^{-1}(x_i)=-y_i$. The formula for bigger degree follows by analyzing, recursively, the identity $(\gamma^{-1}\ast\gamma)(b)=\eps(b)=0$, $b\in\BB^+$.
	For instance:
	\begin{align*}
		0&=\gamma^{-1}(x_0x_1)=\gamma^{-1}(x_0x_1)+\gamma^{-1}(x_0)\gamma(x_1)-\gamma^{-1}(x_2)\gamma(x_0) +\gamma(x_0x_1)\\
		&=\gamma^{-1}(x_0x_1)-y_0y_1+y_2y_0 +y_0y_1-\ul,
	\end{align*}
	so $\gamma^{-1}(x_0x_1)=-y_2y_0 +\ul$. The other values are obtained similarly.
\end{proof}

We recall the characterization of the group $U_H(\fk_3^\ast)$ from Lemma \ref{lem:alpha}, so $[\sigma_{\lambda,\mu}]=\{\alpha_d\ra\sigma_{\lambda,\mu}:d\in\k\}$. We set:
\begin{align*}
	p_{\lambda,\mu}&=(\mu-\ul)^2(\mu+2\ul),  & q_{\lambda,\mu}(d)&=-\mu(6\ul^3-3\ul^2\mu+\mu^3)+d(3\ul^2+\mu^2-d),
\end{align*}
together with $r_{\lambda,\mu}=\ul(\mu-\ul)$.
We present the full description of $[\sigma_{\lambda,\mu}]$; empty slots correspond to the value zero.

\begin{theorem}\label{cociclodiagonal}
	The class $[\sigma_{\lambda,\mu}]$ of the Hopf cocycle associated to $\mE_{\lambda,\mu}$ is given by:
	\begin{table}[H]
		\resizebox{12.5cm}{!}
		{\begin{tabular}{| c|c|c|c|c|c|c|c|c|c|c|c|}
				\hline
				\multirow{2}{*}{$[\sigma_{\lambda,\mu}]$}  & \multirow{2}{*}{$x_0$} & \multirow{2}{*}{$x_1$} & \multirow{2}{*}{$x_2$} & \multirow{2}{*}{$x_0x_1$} & \multirow{2}{*}{$x_0x_2$} & \multirow{2}{*}{$x_1x_0$} & \multirow{2}{*}{$x_1x_2$} & \multirow{2}{*}{$x_0x_1x_0$} & \multirow{2}{*}{$x_0x_1x_2$} & \multirow{2}{*}{$x_1x_0x_2$} & \multirow{2}{*}{$x_0x_1x_0x_2$}\\
				&&&&&&&&&&&\\
				\hline
				$x_0$  & \multirow{2}{*}{$\mu$} & \multirow{2}{*}{$\ul$} & \multirow{2}{*}{$\ul$} &  &  &  &  & \multirow{2}{*}{$-r_{\lambda,\mu}$} & \multirow{2}{*}{$r_{\lambda,\mu}$} & \parbox[t]{1.1cm}{$\ul^2+\mu^2$\\ \hspace*{.5cm} $-d$} & \\
				\hline
				$x_1$  & \multirow{2}{*}{$\ul$} & \multirow{2}{*}{$\mu$} & \multirow{2}{*}{$\ul$} &  &  &  &  & \multirow{2}{*}{$-r_{\lambda,\mu}$} & \parbox[t]{1.1cm}{$\ul^2+\mu^2$\\ \hspace*{.5cm}$-d$} & \multirow{2}{*}{$r_{\lambda,\mu}$} & \\
				\hline
				$x_2$  & \multirow{2}{*}{$\ul$} & \multirow{2}{*}{$\ul$} & \multirow{2}{*}{$\mu$} &  &  &  &  & \parbox[t]{1.3cm}{$-\ul^2-\mu^2$ \\ \hspace*{.8cm}$+d$} & \multirow{2}{*}{$r_{\lambda,\mu}$} & \multirow{2}{*}{$r_{\lambda,\mu}$} & \\
				\hline
				$x_0x_1$  &  &  &  & \multirow{2}{*}{$-2r_{\lambda,\mu}$} & \parbox[t]{.7cm}{$2\ul^2$\\ \hspace*{.1cm} $-d$} & \multirow{2}{*}{$\mu^2-\ul^2$} & \multirow{2}{*}{$r_{\lambda,\mu}$} &  &  &  & \\
				\hline
				$x_0x_2$  &  &  &  & \parbox[t]{.7cm}{$2\ul^2$\\ \hspace*{.1cm} $-d$} & \multirow{2}{*}{$-2r_{\lambda,\mu}$} & \parbox[t]{.7cm}{$r_{\lambda,\mu}$\\ \hspace*{.1cm} $+d$} & \multirow{2}{*}{$-\ul^2-\mu^2$} &  &  &  & \\
				\hline
				$x_1x_0$  &  &  &  & \multirow{2}{*}{$\mu^2-\ul^2$} & \multirow{2}{*}{$r_{\lambda,\mu}$} & \multirow{2}{*}{$-2r_{\lambda,\mu}$} & \parbox[t]{.7cm}{$2\ul^2$\\ \hspace*{.1cm} $-d$} &  &  &  & \\
				\hline
				$x_1x_2$  &  &  &  & \multirow{2}{*}{$r_{\lambda,\mu}$} & \multirow{2}{*}{$-\ul^2-\mu^2$} & \parbox[t]{.7cm}{$2\ul^2$\\ \hspace*{.1cm} $-d$} & \multirow{2}{*}{$-2r_{\lambda,\mu}$} &  &  &  & \\
				\hline
				$x_0x_1x_0$  & \multirow{2}{*}{$-r_{\lambda,\mu}$} & \multirow{2}{*}{$-r_{\lambda,\mu}$} & \parbox[t]{.7cm}{$2\ul^2$\\ \hspace*{.1cm} $-d$} &  &  &  &  & \multirow{2}{*}{$p_{\lambda,\mu}$} &  &  & \\
				\hline	
				$x_0x_1x_2$  & \multirow{2}{*}{$r_{\lambda,\mu}$} & \parbox[t]{.8cm}{$-2\ul^2$\\ \hspace*{.2cm} $+d$} & \multirow{2}{*}{$r_{\lambda,\mu}$} &  &  &   &  &  & \multirow{2}{*}{$p_{\lambda,\mu}$}&  & \\
				\hline
				$x_1x_0x_2$  & \parbox[t]{.8cm}{$-2\ul^2$\\ \hspace*{.2cm} $+d$} & \multirow{2}{*}{$r_{\lambda,\mu}$} & \multirow{2}{*}{$r_{\lambda,\mu}$} &  &  &   &  &  &  & \multirow{2}{*}{$p_{\lambda,\mu}$} & \\
				\hline
				\parbox[t]{1.5cm}{$x_0x_1x_0x_2$ \\}  &  &  &  &  &  &  &  &  &  &  & \multirow{2}{*}{$q_{\lambda,\mu}(d)$} \\
				\hline
		\end{tabular}}
	\end{table}
\end{theorem}

\begin{proof}
	We use Lemma \ref{lem:gamma-pointed} and compute the cocycle $\sigma_{\lambda,\mu}$ using \gap, following \eqref{eqn:gap-cocycle}.
	In turn, we follow \eqref{eqn:gap-orbit} to write down the orbit $\alpha_d\ra\sigma_{\lambda,\mu}$, $d\in\k$.  
	
	This is done \texttt{cocycle-pointed.g}. 
\end{proof}

\begin{remark}
	As an alternative proof, we compute by hand the set 
	\[
	S_\sigma=\{\sigma(x_i,b):i\in X, b\in \mathbb{B}\}
	\]
	and finish the proof following the decomposition formula from \cite[Lemma 3.2]{GS}: these artisanal computations will be part of \cite{Stesis}. As an example, We recall from \cite[Example 3.3]{GS} that in degree 2 this formula reads:
	\begin{align}\label{eqn:dec-form2}
		\sigma(x_ix_j,b)=\sigma(x_i,x_jb)+\sigma(x_i,b_{\underline{(2)}})\sigma(x_j,b_{\underline{(1)}}).
	\end{align}
\end{remark}

\subsection{Hochschild cocycles}\label{subsec-hoch-pointed}

We recall the cocycles $\xi_i^j\in \Z^2(\fk_3,\k)$, $i,j\in X$, defined in \eqref{eqn:etabi} and let
\begin{align*}
	\xi_0&\coloneqq \sum_{i\in X}\xi_i^i, & \xi_2&\coloneqq \sum_{i\neq j\in X}\xi_i^j.
\end{align*}
We introduce a third linear map $\xi_{3}\colon \fk_3\ot \fk_3\to \k$. To do this, we first fix subsets $X_\pm\subset \BB\times \BB$ as 
\begin{align*}
	X_+&=\{(x_2,x_0x_1x_0),(x_0x_2,x_1x_2),(x_1x_2,x_0x_2),(x_1x_0x_2, x_0),(x_0x_1x_2,x_1)\},\\ 
	X_-&=\{(x_0,x_1x_0x_2),(x_1,x_0x_1x_2), (x_0x_1,x_0x_2),(x_1x_0,x_1x_2),\\
	&\hspace*{5cm} (x_0x_2,x_0x_1),(x_1x_2,x_1x_0),(x_0x_1x_0,x_2)\}
\end{align*}
and define 
\begin{align*}
	\xi_{3}(a,b)&=\pm1, \ \text{if }(a,b)\in X_\pm, & \xi_{3}(a,b)&=0 \ \text{elsewhere in }\BB\times \BB.
\end{align*}

Once again, we remind of the braiding \eqref{eqn:braiding-intro} and its  realization \eqref{eqn:real-k}, which determines the $\G_{3,\ell}$-action; as well as the elements $g_i\in\G_{3,\ell}$, $i\in X$.
\begin{lemma}\label{lem:cocycles-xib}
	The following assertions hold.
	\begin{enumerate}
		\item[(a)] $\xi_0, \xi_2,\xi_{3}\in \Z^2(\fk_3,\k)^{\mathbb{G}_{3,\ell}}$.
		\item[(b)] $\xi_2,\xi_{3}\in  \B^2(\fk_3,\k)$.
	\end{enumerate}
\end{lemma}
\pf
{\it (a)} Direct calculations show that all three maps $\xi_0, \xi_2$ and $\xi_{3}$ verify \eqref{eqn:conds-hoch}; hence they belong to $\in \Z^2(\fk_3,\k)$.
As well, it is clear that $\xi_0$ and $\xi_2$ are $\mathbb{G}_{3,\ell}$-linear.

As for $\eta=\xi_3$, it is sufficient to show that the linearity is fulfilled in values of the form $\eta(x_ix_jx_k,x_j)$ with $i,j,k\in X$, all different. We get:
\begin{align*}
	\eta(g_i\cdot (x_ix_jx_k),g_i\cdot x_j)&=\eta(x_{i\triangleright i}x_{i\triangleright j}x_{i\triangleright k},x_{i\triangleright j})=\eta(x_{i}x_{k}x_{j},x_{k}),
\end{align*}
same for $g_j, g_k$. 
In this way we see that $\xi_3\in \Z^2(\fk_3,\k)^{\mathbb{G}_{3,\ell}}$ as well.

{\it (b)} We see that $\xi_2$ and $\xi_3$ are coboundaries by considering the linear maps $f_2,f_3:\fk_3\to \k$ defined on the basis $\mathbb{B}$ as 
\[
f_2(x_ix_j)=-\xi_2(x_i,x_j), \qquad f_3(x_0x_1x_0x_2)=\xi_3(x_1x_0x_2,x_0),
\] and zero elsewhere.
\epf

Next result shows that the cocycles from Lemma \ref{lem:cocycles-xib}, together with the counit map $\veps$, generate all $\mathbb{G}_{3,\ell}$-invariant cocycles in $\Z^2(\fk_3,\k)$.

\begin{proposition}\label{lemprop:hochschild-pointed}
	$\Z^2(\fk_3,\k)^{\mathbb{G}_{3,\ell}}$ is generated  over $\k$ by $\veps,\xi_0,\xi_2,\xi_3$.
\end{proposition}
\pf
Let $\eta\in \Z^2(\fk_3,\k)^{\mathbb{G}_{3,\ell}}$ be a $\mathbb{G}_{3,\ell}$-linear 2-cocycle. Recall from \S\ref{subsec:hoch2} that $\eta$ is determined by the set of values $\{\eta(b,x_i)\colon b\in\mathbb{B}, i\in X\}\cup\{\eta(1,1)\}$. Moreover, as $\mathbb{G}_{3,\ell}\cdot x_0=\{x_0,x_1,x_2\}$, then in this case $\eta$ is actually determined by  $\eta(1,1)$ and by the other twelve scalars
$\eta(b,x_0)$, $b\in \mathbb{B}$.
We claim that
\begin{enumerate}
	\item $\eta(x_0,x_0)=\eta(x_1,x_1)=\eta(x_2,x_2)$ and $\eta(x_1,x_0)=\eta(x_2,x_0)$,
	\item $\eta(x_1x_0x_2,x_0)=\eta(x_0x_1x_2,x_1)=-\eta(x_0x_1x_0,x_2)$.
	\item $\eta$ is zero for the remaining cases.
\end{enumerate}
The first two items are a direct consequence of $\mathbb{G}_{3,\ell}$-linearity of $\eta$ and relations in $\fk_3$. To see the third, we use that $\eta$ is a Hochschild 2-cocycle, the relations in $\fk_3$ and $\mathbb{G}_{3,\ell}$-linearity. First, notice that  $\eta(x_ix_j,x_k)= 0$ for $j=k$, so we suppose that $\eta(x_ix_j,x_k)\neq 0$; then:
\[-\eta(x_kx_i,x_k)=\eta(x_ix_j,x_k)=\eta(x_i,x_jx_k)=-\eta(x_i,x_kx_i)=-\eta(x_ix_k,x_i);\]
but the $\mathbb{G}_{3,\ell}$-linearity says that $\eta(x_ix_k,x_i)=-\eta(x_kx_i,x_k)$. Hence $\eta(x_ix_j,x_k)=0$. On the other hand, $\eta(x_ix_jx_i,x_k)=0$ for $k=i,j$. Finally, it is clear that $\eta(x,y)=0$ for $\deg x+\deg y\geq 5$.
Thus, $\eta$ is just determined by the scalars
\begin{align*}
	&\eta(1,1),&\eta(x_2,&x_0),& \eta(x_0,&x_0),& \eta(x_0&x_1x_0,x_2).
\end{align*}
We get the conclusion from  the definition of the cocycles in Lemma \ref{lem:cocycles-xib}.
\epf

\subsubsection{Exponentials}

We characterize Hochschild 2-cocycles $\eta \in \Z^2(\fk_3,\k)^{\mathbb {G}_{3,\ell}}$ such that $e^{\eta}$ is a Hopf 2-cocycle. 
We recall from \cite[Lemma 5.3]{GS} that we can assume that $\eta(1,1)=0$. So, we can restrict to cocycles of the form
\begin{align}\label{eq:eta-fk3diagonal}
	\eta=\eta_0\xi_0+\eta_2\xi_2+\eta_{3}\xi_{3},
\end{align}
for the scalars $\eta_i\coloneqq\eta(x_i,x_0)$, $i=0,2$ and $\eta_{3}\coloneqq\eta(x_0x_1x_0,x_2)$.

Recall that a sufficient condition is for $\eta$ to verify \eqref{eqn:conm1}. We shall show in Proposition \ref{pro:fk3-converse}, where a complete characterization of these cocycles is established, that this is not necessary. Indeed, it is easy to obtain a necessary condition for $\eta$ to verify the aforementioned conditions:
\begin{lemma}
	Let $\eta\in \Z^2(\fk_3,\k)^{\mathbb{G}_3,\ell}$ be such that it satisfies \eqref{eqn:conm1}. Then 
	\[
	\eta\in \C\coloneqq\{\eta_{0}(\xi_0+\xi_2),\, \eta_{3}\xi_{3}:\eta_0,\eta_{3}\in\k\}.
	\]
\end{lemma} 
We mention that the converse is also true:  when $\eta=\xi_0+\xi_2$, then $e^\eta$ is a Hopf cocycle by \cite[Lemma 2.3.]{GaM}. However, the alternative $\eta=\xi_3$ requires a lengthy case-by-case analysis that we choose to omit.
\pf
Suppose that $\eta$ verifies \eqref{eqn:conm1}. In particular, the first identity is
\begin{align}\label{eq:etaconm1}
	\eta(x,y_{(1)}z_{(1)})\eta(y_{(2)},z_{(2)})=\eta(x,y_{(2)}z_{(2)})\eta(y_{(1)},z_{(1)}),
\end{align}
for any $x,y,z\in \fk_3$. Next we distinguish two cases:  $\eta_{0}=\eta(x_2,x_0)\neq 0$ and $\eta_{0}=\eta(x_2,x_0)=0$.

If $\eta(x_2,x_0)\neq 0$, then we set $y=x_2x_0$, $z=x_0$ and from \eqref{eq:etaconm1} we get
$\eta(x,x_1)=\eta(x,x_0)$. Now if we put $x=x_1x_2x_1$, then $\eta(x_1x_2x_1,x_0)=0$, so by $H$-linearity  $\eta(x_0x_1x_0,x_2)=0$. Instead, if $x=x_0$, then $\eta(x_0,x_1)=\eta(x_0,x_0)$, moreover $\eta(x_2,x_0)=\eta(x_0,x_0)$.
Thus, $\eta=\eta_{0}(\xi_0+\xi_2)$.



On the other hand, assume $\eta(x_2,x_0)=0$. Let $i,j,k\in X$ be all different and set $y=x_ix_j$,  $z=x_i$, then \eqref{eq:etaconm1} is
$\eta(x,x_k)\eta(x_i,x_i)=0$. In particular, if $x=x_k$, then this gives $\eta\in\k\{\xi_3\}$.
\epf

	Next we give a full characterization of Hochschild cocycles that can be exponentiatied to produce a Hopf cocycle.
	
	\begin{proposition}\label{pro:fk3-converse}
		Fix $\eta\in \Z^2(\fk_3,\k)^{\G_{3,\ell}}$. Then $e^\eta\in Z^2(\fk_3)$ if and only if
		\begin{align*}
			\eta \in \overline{\C}\coloneqq\{\eta_0(\xi_0+\xi_2)+\eta_{3}\xi_{3} :\eta_0,\eta_{3}\in\k\}.
		\end{align*}
	\end{proposition}
	
	We extract from the proof a technical lemma, for a cleaner exposition.
	\begin{lemma}\label{lem:tecnical-p}
		Let $\eta\in\overline{\C}$. If $a,b\in\mathbb{B}^+$, then 
		\begin{align*}
			e^{\eta}(a,b)=\begin{cases}
				\eta(a,b), & \text{if }\deg(a)+\deg(b)< 6,\\
				\frac{1}{2}\eta^{\ast 2}(a,b), & \text{if }\deg(a)+\deg(b)\geq  6.
			\end{cases}
		\end{align*}
		Moreover, if $\eta=\eta_0(\xi_0+\xi_2)+\eta_{3}\xi_{3}$, then it is given by the table:
		\begin{center}
			\begin{table}[H]
				\resizebox{12cm}{!}
				{\begin{tabular}{|c|c|c|c|c|c|c|c|c|c|c|c|}
						\hline
						$e^{\eta}$  & $x_0$ & $x_1$ & $x_2$ & $x_0x_1$ & $x_0x_2$& $x_1x_0$& $x_1x_2$ & $x_0x_1x_0$ & $x_0x_1x_2$ & $x_1x_0x_2$ & $x_0x_1x_0x_2$\\
						\hline
						$x_0$  & $\eta_0$ & $\eta_0$ & $\eta_0$ &  &  & &  &  &  & $\eta_{3}$ & \\
						\hline
						$x_1$  & $\eta_0$ & $\eta_0$ & $\eta_0$ &  &  & &  &  & $\eta_{3}$ &  & \\
						\hline
						$x_2$  & $\eta_0$ & $\eta_0$ & $\eta_0$ &  &  & &  & $-\eta_{3}$ &  &  & \\
						\hline
						$x_0x_1$  &  &  &  &  & $\eta_{3}$ &  &  &  &  &  & \\
						\hline
						$x_0x_2$  &  &  &  & $\eta_{3}$ &  & & $-\eta_{3}$ &  &  &  & \\
						\hline
						$x_1x_0$  &  &  &  &  &  & & $\eta_{3}$ &  &  &  & \\
						\hline
						$x_1x_2$  &  &  &  &  & $-\eta_{3}$ & $\eta_{3}$ &  &  &  &  & \\
						\hline
						$x_0x_1x_0$  &  &  & $\eta_{3}$ &  &  & &  &  &  &  & \\
						\hline
						$x_0x_1x_2$ &  & $-\eta_{3}$ &  &  &  & &  &  &  &  & \\
						\hline
						$x_1x_0x_2$  & $-\eta_{3}$ &  &  &  &  &  &  &  &  &  & \\
						\hline
						$x_1x_0x_1x_2$  &  &  &  &  &  & &  &  &  &  & $-\eta_{3}^2$\\
						\hline
				\end{tabular}}
			\end{table} 
		\end{center}
	\end{lemma}

	\begin{remark}\label{rem:sigma-exp}
		If $\eta=\eta_0(\xi_0+\xi_2)+2\eta_{0}^2\xi_{3}$, then $e^\eta=\sigma_{3\eta_0,\eta_0}$. This follows by comparing the table in the Lemma with that of Theorem \ref{cociclodiagonal} (for $d=0$).
		In particular, this exponential $e^\eta$ defines a Hopf cocycle.
	\end{remark}
	
	\pf
	We start with a numerical observation: we have that $e^{\eta}(a,b)=\eta(a,b)$ when $\deg a+\deg b$ is neither $6$ nor $8$. Indeed, for $B=\fk_3$ we have $B=\bigoplus_{n= 0}^4B_n$ as a graded Hopf algebra, so \cite[Remark 5.4]{GS} implies
	$e^{\eta}= \epsilon+\eta+\frac{1}{2!}\eta^{\ast 2}+\frac{1}{3!}\eta^{\ast 3}+\frac{1}{4!}\eta^{\ast 4}$.
	
	Moreover, notice that, for $1\leq n,m\leq 4$:
	\begin{itemize}
		\item $\eta_{|B_n\ot B_m}=0$ if $(n,m)\notin\{(1,1),(1,3),(2,2),(3,1)\}$.
	\end{itemize}
	Hence, we can conclude that:
	\begin{itemize}
		\item $\eta^{\ast 2}|_{B_n\ot B_m}=0$ if $(n,m)\notin\{(2,2),(2,4),(4,2),(3,3)\}$.
		\item $\eta^{\ast 3}|_{B_n\ot B_m}=0$ if $(n,m)\notin\{(2,4),(4,2),(3,3), (4,4)\}$.
		\item $\eta^{\ast 4}|_{B_n\ot B_m}=0$ if $(n,m)\neq(4,4)$.
	\end{itemize}
	Indeed, the argument of these conclusions is the same in each case, so we show only for $\eta^{\ast 2}|_{B_n\ot B_m}=0$. Notice that
	\begin{align*}
		\Delta^{2}(B_n\ot B_m)\subseteq\sum_{\substack{n_1+n_2=n\\
				m_1+m_2=m}} B_{n_1}\ot B_{m_1}\ot B_{n_2}\ot B_{m_2}
	\end{align*}  
	Let $a\in B_n$ and $b\in B_m$ be two element of positive degree, we know that $\eta^{\ast 2}(a,b)=\eta(a_{(1)}, b_{(1)})\eta(a_{(2)}, b_{(2)})$ has possibly non-zero values only when $a_{(1)}\ot b_{(1)}\ot a_{(2)}\ot b_{(2)}$ is an element of $B_{s}\ot B_{t}\ot B_{v}\ot B_{w}$ where 
	$(s,t,v,w)$ is in 
	\[\{(1,1,1,1),(1,1,1,3),(1,1,2,2), (1,1,3,1),(1,3,1,1),(2,2,1,1),(3,1,1,1)\}.\]
	This can only occur when $(n,m)\in\{(2,2),(2,4),(4,2),(3,3)\}$.
	
	\medbreak 
	Now observe that when $\deg a=\deg b=2$, it follows $\underline{\Delta}(a)$ and $\underline{\Delta}(b)$ are alternating sums of terms of the form $\pm x_i\ot x_j$, $i\neq j\in X$, which gives $\eta^{\ast 2}(a,b)=0$. 
	
	We remark that a quick look at \eqref{eqn:coprod1} shows that this propagates to higher degrees: namely that for $a$ of degree $n$ we have $\underline{\Delta}^{(n-1)}(a)\in B_1^{\ot n}$ is an alternating sum as well. Hence $\eta^{\ast n}(a,b)=0$ when $n=\deg a=\deg b\in\{3,4\}$.
	
	Hence we need to deal with the cases $(\deg a,\deg b)\in \{(2,4),(4,2),(4,4)\}$.
	
	When 
	$(\deg a,\deg b)\in \{(2,4),(4,2)\}$, we have that $e^{\eta}(a,b)=\frac12\eta^{\ast2}(a,b)$ by definition and Remark \cite[Remark 5.4]{GS}.
	
	
	
	Finally we verify $\eta^{\ast 3}(a,b)=0$ for $a=b=x_0x_1x_0x_2$:  for this, we show that $\eta^{\ot 3}\pi_i\Delta^2(a\ot b)=0$, for each one of the linear projections 
	$\pi_1\colon B^{\ot 6}\to B_2^{\ot 2}\ot B_1^{\ot 2}\ot B_1^{\ot 2}$, 
	$\pi_2\colon B^{\ot 6}\to B_1^{\ot 2}\ot B_2^{\ot 2}\ot B_1^{\ot 2}$ and 
	$\pi_3\colon B^{\ot 6}\to B_1^{\ot 2}\ot B_1^{\ot 2}\ot B_2^{\ot 2}$ onto the support of $\eta^{\ot 3}$.
	Indeed, for $\pi_1$ we obtain a sum of expressions of the form $\eta(x_ix_j,x_lx_m)\eta_0^2-\eta(x_ix_j,x_lx_m)\eta_0^2$ with $i\neq j,l\neq m\in X$. Next, for $\pi_2$ we obtain:
	\begin{align*}
		\eta(a_{(1)},&b_{(1)})\eta(a_{(2)},b_{(2)})\eta(a_{(3)}, b_{(3)})=\eta_0^2[-\eta(x_1x_0,x_2x_0)-\eta(x_1x_0,x_1x_2)\\	&-\eta(x_2x_0,x_2x_1)-\eta(x_0x_1,x_2x_1)+\eta(x_0x_1,x_2x_1)\\	&-\eta(x_2x_1,x_2x_0)+\eta(x_2x_1,x_2x_0)-\eta(x_2x_1,x_0x_1)\\
		&+\eta(x_0x_1,x_0x_2)-\eta(x_0x_2,x_1x_2)-\eta(x_0x_2,x_0x_1)\\
		&+\eta(x_0x_2,x_1x_2)+\eta(x_2x_0,x_1x_0)-\eta(x_2x_0,x_1x_0)]=0.
	\end{align*}
	The computation for $\pi_3$ is similar to that of $\pi_2$. This shows the first claim.
	
	\medbreak
	
	The table now follows by checking that $\eta^{\ast 2}(a,b)=0$ whenever $a$ or $b$ are different from the longest element $x_0x_1x_0x_2$, where the result is $-2\eta_3^2$.
	This is a tedious by-hand computation, that can also be checked with \texttt{eta2-pointed.g} using $\gap$.
	\epf
	
	Next we prove Proposition \ref{pro:fk3-converse}.
	\pf
	Let $\eta\in \Z^2(\fk_3,\k)^{\G_{3,\ell}}$ be as in \eqref{eq:eta-fk3diagonal} and assume that $e^\eta$ is a Hopf 2-cocycle. 
	Notice that $e^\eta(x_{1}x_{0},x_{1}x_{0})=\frac{1}{2}(-\eta_2\eta_0+\eta_2\eta_2)$ by definition while we get  $e^\eta(x_{1}x_{0},x_{1}x_{0})=-\eta_2\eta_0+\eta_2\eta_2$ from the decomposition formula \eqref{eqn:dec-form2}. Then $\eta_2(\eta_2-\eta_0)=0$.  
	Similarly, from $e^\eta(x_{1}x_{0},x_{2}x_{0})$ we get $\frac{1}{2}(-\eta_0\eta_0+\eta_2\eta_2)=0$. 
	We conclude that $\eta_0=\eta_2$ and therefore $\eta\in \overline{\C}$.
	
	Conversely, let us fix $\eta=\eta_0(\xi_0+\xi_2)+\eta_{3}\xi_{3}\in\overline{\C}$. If $\eta_3=2\eta_0^2$, then $e^\eta$ is the Hopf cocycle $\sigma_{3\eta_0,\eta_0}$, as seen in Remark \ref{rem:sigma-exp}. For the general case, we observe that:
	\[
	e^\eta=\alpha_{d}\rightharpoonup\sigma_{3\eta_0,\eta_0}\in Z^2(\fk_3), \qquad d=2\eta_0^2-\eta_{3},
	\]
	for $\alpha_d\in U_H(\fk_3^{\ast})$ as in Lemma \ref{lem:alpha}. Indeed, this follows by comparing the table for $e^\eta$ in Lemma \ref{lem:tecnical-p} and the orbit 
	$\alpha_d\ra\sigma_{3\eta_0,\eta_0}$ as computed in Theorem \ref{cociclodiagonal}. 
	\epf

	\begin{remark}\label{rem:inclusion-strict-pointed}
		We have observed that $\C\subset \overline{\C}$ and this inclusion is strict. However, we have also seen that $\xi_2,\xi_3\in \B^2(\fk_3,\k)$, so $\eta_0\xi_2+\eta_{3}\xi_{3}\in \B^2(\fk_3,\k)$. Thus $\C\equiv\overline{\C}$ up to coboundaries. 
	\end{remark}

	Now we can rapidly characterize pure Hopf cocycles.
	
	\begin{theorem}\label{thm:fk3-sigma-pure}
		Let $\sigma_{\lambda,\mu}$ be a 2-cocycle as in Theorem \ref{cociclodiagonal}. Then the following are equivalent:
		\begin{enumerate}[label=(\alph*)]
			\item $\sigma_{\lambda,\mu}\sim e^\eta$ for some $\eta\in \Z^2(\fk_3,\k)^{\mathbb{G}_3,\ell}$.
			\item $\mu=\ul$.
		\end{enumerate}
	\end{theorem}
	\pf 
	{\it (a)} $\Rightarrow${\it (b)}. Let $d\in\k$ and $\eta\in \Z^2(\fk_3,\k)^{\mathbb{G}_3,\ell}$ be such that $\sigma_{\lambda,\mu}=\alpha_d\ra e^\eta$. 
	In particular $e^\eta=\alpha_{-d}\ra\sigma_{\lambda,\mu}$ is a Hopf cocycle, namely $\eta\in\overline{\C}$ and thus there are $\eta_0,\eta_3\in \k$ such that $\eta=\eta_0(\xi_0+\xi_2)+\eta_3\xi_3$. 
	By comparing the tables for $\alpha_{-d}\ra\sigma_{\lambda,\mu}$ in Theorem \ref{cociclodiagonal} and for $e^\eta$ in Lemma  \ref{lem:tecnical-p}, it follows that
	$\ul=\mu=\eta_0$ (and $d=\eta_3-2\eta_0^2$).
	
	{\it (b)} $\Rightarrow${\it (a)}. This is Remark \ref{rem:sigma-exp}, as $\sigma_{\lambda,\mu}=\sigma_{\lambda,\ul}=e^\eta$, for $\eta=\ul(\xi_0+\xi_2)+2\ul^2\xi_{3}$.
	\epf

	The following remark deals with the apparent difference between the sets $\C$ and $\overline{\C}$ of Hochschild cocycles leading to Hopf cocycles under exponentiation. This motivates Question \ref{q}, along with similar results in \cite{GS} and the next section.

	\begin{remark}\label{rem:question-p}
		It follows from the proof of Theorem \ref{cociclodiagonal} that $\ex(\C)=\ex(\overline{\C})$ in $H^2(\fk_3)$. 
		Indeed, if $\eta=\eta_0(\xi_0+\xi_2)+\eta_3\xi_3\in\overline{\C}$, then $e^\eta=\alpha_{d}\ra\sigma_{3\eta_0,\eta_0}$, for $d=2\eta_0^2-\eta_3$. Now, $\alpha_{\eta_3}\ra e^\eta=e^\xi$, for $\xi=\eta_0(\xi_0+\xi_2)\in \C$. The reverse inclusion is automatic.
	\end{remark}
	
	Now we look at these results in the Hopf algebra setting, to investigate the range of exponential cocycles, or the ubiquity of pure ones. When $\ell=1$, then exponential cocycles suffice, see Corollary \ref{cor:s3-intro}. The situation changes for bigger $\ell$.
	
	\begin{corollary}\label{cor:pointed-exponential}
		Assume $\ell>1$. Then $A_{\lambda,\mu}$ is isomorphic to a nontrivial exponential cocycle deformation of $A=\fk_3\#\k\G_{3,\ell}$ only if $[\lambda:\mu]=[3:1]$ in $\mathbb{P}^1$.
	\end{corollary}
	\pf
	Let $A=\fk_3\#\k\G_{3,\ell}$ and $\sigma=e^\eta\#\eps\in Z^2(A)$, for some $\eta\in\Z^2(\fk_3,\k)^{\G_{3,\ell}}$. In particular, $\eta\in\overline{\C}$ by Proposition \ref{pro:fk3-converse} and $\sigma\sim\sigma_{\lambda,\mu}$ with $\mu=\lambda/3$ by Theorem \ref{thm:fk3-sigma-pure}. Hence $A_{\lambda,\mu}\simeq A_{3,1}$.
	\epf

\section{Deformations over $\k^{\s_3}$}\label{sec:copointed}

In this section we fix $H=\k^{\s_3}$ and compute the Hopf 2-cocycles associated to the deformations of $A=\fk_3\# H$; recall realization \eqref{eqn:coreal}. We follow the lines of Section \ref{sec:pointed}, {\it mutatis mutandis}. Once again we recall  relation \eqref{RelA}.

\subsection{Hopf cocycles}

The cleft objects for the Nichols algebra $\fk_3$ in $H\mod$ are the algebras $\mE_{{\bf c}}$, ${\bf c}\coloneqq(c_0,c_1,c_2)$, generated by $y_0,y_1,y_2$ with relations:
\begin{align*}
	y_i^2&=c_i,\,i\in X, & 
	y_1y_0+y_0y_2+y_2y_1=y_0y_1+y_1y_2+y_2y_0&=0.
\end{align*}
See \cite{GIV1} for details. By construction, this algebra inherits a basis $\mathbb{B}$ as in \eqref{eqn:basis}.
In this setting, we observe the following relations, for $i,j,k\in X$ pairwise different:
\begin{align}\label{s3-relH2}
	y_iy_ky_j&=y_jy_ky_i, &y_iy_jy_i&=y_jy_iy_j+c_jy_k-c_iy_k=-y_jy_ky_i-c_iy_k.
\end{align}
Once again, $\mE_{{\bf c}}$ is a (right) $\fk_3$-comodule algebra  with coaction $\rho$ as in \eqref{eqn:coaction}.

We start by computing the corresponding section $\gamma_{\cc}\colon \fk_3\to\mE_{{\bf c}}$.

\begin{lemma}\label{lem:gamma-copointed}
	Let $i,j,k\in X$ be all different.
	\begin{enumerate}[leftmargin=*,label=(\alph*)]
		\item  The section $\gamma\coloneqq\gamma_{\cc}\colon \fk_3\to \mE_{{\bf c}}$ is determined by:
		\begin{align*}
			\gamma(x_i)&=y_{i}, & \gamma(x_{i}x_{j})&= y_{i}y_{j}, &
			\gamma(x_0x_1x_0)&=y_0y_1y_0+c_0y_2, \\
			\gamma(x_0x_1x_2)&=y_0y_1y_2, & \gamma(x_1x_0x_2)&=y_1y_0y_2, &
			\gamma(x_0x_1x_0x_2)&=y_0y_1y_0y_2.
		\end{align*}
		\item The inverse $\gamma^{-1}:\fk_3\to \mE_{{\bf c}}$ is given by 
		\begin{align*}
			\gamma^{-1}(x_ix_j)&=-y_{j}y_{k}, & \gamma^{-1}(x_0x_1x_0x_2)&=y_{0}y_{1}y_{0}y_{2}+c_{2}c_{0}-c_{2}c_{1}+c_{0}c_{1},
		\end{align*}
		and $\gamma^{-1}(b)=-\gamma(b)$, for $b\in\BB$ with positive odd degree.
	\end{enumerate}
\end{lemma}

\begin{proof}
	(a)  We begin with the computation of the coproduct of every $b\in \mathbb{B}$. Then we define $\gamma(b)$ so that $\rho\circ\gamma(b)=(\gamma\ot \id)\circ\Delta(b)$.
	Let $i,j,k\in X$ be all different:
	\begin{align}
		\notag \underline{\Delta}(x_ix_j)=&x_i\ot x_j-x_j\ot x_k.
		\\
		\label{eqn:coprod2}\underline{\Delta}(x_ix_jx_i)=&
		x_jx_i\ot x_j+x_i\ot x_ix_{k}+x_ix_j\ot x_i+x_i\ot x_jx_i-x_j\ot x_{k}x_i.
		\\
		\notag \underline{\Delta}(x_0x_1x_0x_2)=&x_0x_1x_0\ot x_{2}-x_0x_1x_2\ot x_{1}-x_1x_0x_2\ot x_{0}+x_0\ot x_{1}x_{0}x_{2}
		\\
		\notag &+x_1\ot x_{0}x_{1}x_{2}-x_2\ot x_{0}x_{1}x_{0}+x_0x_1\ot x_{0}x_{2}+x_0x_2\ot x_{0}x_{1}\\
		\notag &+x_1x_0\ot x_{1}x_{2}+x_1x_2\ot x_{1}x_{0}+x_1x_2\ot x_{0}x_{2}+x_0x_2\ot x_{1}x_{2}.
	\end{align}
	
	It readily follows that $\gamma(x_i)=y_i$,  and $\gamma(x_ix_j)=y_iy_j$.
	Notice that \eqref{RelA} gives $\Delta(x_ix_jx_i)=\Delta(x_jx_ix_j)=-\Delta(x_ix_kx_j)$.
	On the other hand, we have that
	\[
	\rho(y_iy_jy_i)-(\gamma\ot \id)\Delta(x_ix_jx_i)=y_iy_jy_i\ot 1-c_i\ot x_{k}-\gamma(x_ix_jx_i)\ot 1,
	\]
	hence, by \eqref{s3-relH2}:
	$\gamma(x_ix_jx_i)=y_iy_jy_i+c_iy_{k}=-y_iy_ky_j$.
	
	Recall as well that $x_2x_1x_2=-x_1x_0x_2$ and $x_2x_0x_2=-x_0x_1x_2$, which gives  $\gamma(x_1x_0x_2)=y_1y_0y_2$ and $\gamma(x_0x_1x_2)=y_0y_1y_2$.
	Finally, a straightforward calculation gives $\rho(y_0y_1y_0y_2)-(\gamma\ot \id)\Delta(x_0x_1x_0x_2)=0$. Therefore $\gamma(x_0x_1x_0x_2)=y_0y_1y_0y_2$.
	
	It is readily seen that $\gamma$ is $\k^{\s_3}$-linear.
	
	(b) The computation of $\gamma^{-1}$ follows as in Lemma \ref{lem:gamma-pointed}.
\end{proof}

As in the pointed case, the following result follows using \gap. We refer, as well, to \cite{Stesis} for a list of computations by hand. For $\cc=(c_0,c_1,c_2)$, we recall the notation $\alpha_1=c_0-c_2$, $\alpha_2=c_0-c_1$. Again, empty slots correspond to zero. We set
\begin{align*}
	p_{\cc}(d)&\coloneqq \alpha_1\alpha_2-c_0^2+d, & q_{\cc}(d)&=-c_0^2c_1c_2-dp_{\cc}(d).
\end{align*}
\begin{theorem}\label{cocicloCOdiagonal}
	The class $[\sigma_{{\bf c}}]$ of the Hopf cocycle associated to $\mE_{{\bf c}}$ is given by:
	\begin{center}\begin{table}[H]
			\resizebox{12cm}{!}
			{\begin{tabular}{|c|c|c|c|c|c|c|c|c|c|c|c|}
					\hline
					\multirow{2}{*}{$[\sigma_{\cc}]$}  & \multirow{2}{*}{$x_0$} & \multirow{2}{*}{$x_1$} & \multirow{2}{*}{$x_2$} & \multirow{2}{*}{$x_0x_1$} & \multirow{2}{*}{$x_0x_2$} & \multirow{2}{*}{$x_1x_0$} & \multirow{2}{*}{$x_1x_2$} & \multirow{2}{*}{$x_0x_1x_0$} & \multirow{2}{*}{$x_0x_1x_2$} & \multirow{2}{*}{$x_1x_0x_2$} & \multirow{2}{*}{$x_0x_1x_0x_2$}\\
					&&&&&&&&&&&\\
					\hline
					\parbox[t]{.4cm}{$x_0$\\ } & \multirow{2}{*}{$c_0$} &  &  &  & &  &  &  &  &  \multirow{2}{*}{$-d$} & \\
					\hline
					$x_1$ &  & \multirow{2}{*}{$c_1$} &  &  &  &  &  &   & \parbox[t]{1cm}{$c_2\alpha_1$\\ \hspace*{.4cm} $-d$} &  &  \\
					\hline
					$x_2$ &  &  & \multirow{2}{*}{$c_2$} &  &  &  &  & \parbox[t]{1cm}{$-c_1\alpha_1$\\ \hspace*{.4cm} $+d$} &  &  & \\
					\hline
					\parbox[t]{.7cm}{$x_0x_1$\\} &  &  &  &  & \multirow{2}{*}{$-d$} & \multirow{2}{*}{$c_0c_1$} &  &  &  &  & \\
					\hline
					$x_0x_2$ &  &  &  &  &  &  & \parbox[t]{1cm}{$-c_0c_2$\\ \hspace*{.3cm} $+d$} &  &  &  & \\
					\hline
					$x_1x_0$ &  &  &  & \multirow{2}{*}{$c_0c_1$} &  &  & \parbox[t]{1cm}{$c_2\alpha_1$\\ \hspace*{.3cm} $-d$} &  &  &  & \\
					\hline
					$x_1x_2$ &  &  &  &  & \parbox[t]{1cm}{$-c_0c_2$\\ \hspace*{.2cm} $+d$} & \parbox[t]{1cm}{$c_2\alpha_1$\\ \hspace*{.2cm} $-d$} &  &  &  &  & \\
					\hline
					$x_0x_1x_0$ &  &  & \parbox[t]{1cm}{$c_0c_2$\\ \hspace*{.2cm} $-d$} &  &  &  &  & \multirow{2}{*}{$c_0c_1c_2$} &  &  & \\
					\hline
					$x_0x_1x_2$ &  & \parbox[t]{1cm}{$-c_0c_1$\\ \hspace*{.2cm} $+d$} &  &  &   &  &  &  & \multirow{2}{*}{$c_0c_1c_2$} &  & \\
					\hline
					\parbox[t]{1.1cm}{$x_1x_0x_2$ \\ } & \multirow{2}{*}{$p_{\cc}(d)$} &  &  &  &  &  &  &  &  & \multirow{2}{*}{$c_0c_1c_2$} & \\
					\hline
					\parbox[t]{1.5cm}{$x_0x_1x_0x_2$ \\}&  &  &  &  &   &  &  &  &  &  & \multirow{2}{*}{$q_{\cc}(d)$}\\
					\hline
			\end{tabular}}
		\end{table}
	\end{center}
\end{theorem}
\pf
Follows as Theorem \ref{cociclodiagonal}: we compute the cocycle $\sigma_{\cc}$ using \gap, as described in \eqref{eqn:gap-cocycle}; we use the definition of $\gamma_{\cc}$ from Lemma \ref{lem:gamma-copointed} above. 
As for the orbit $\alpha_d\ra\sigma_{\cc}$, $d\in\k$, we follow the recipe described in \eqref{eqn:gap-orbit}.

These computations are performed in \texttt{cocycle-copointed.g}. 
\epf

\subsection{Hochschild cocycles}\label{subsec-hoch-copointed}

Recall from \eqref{eqn:etabi} the definition of the Hochschild 2-cocycles
$\xi_j^i\in \Z^2(\fk_3,\k)$, $i,j\in X$ and the cocycle $\xi_3\in \B^2(\fk_3,\k)$ from Lemma \ref{lem:cocycles-xib}.
It is easy to see that $\xi_j^i$ is a coboundary if and only if $i\neq j$.

\begin{proposition}
	$\Z^2(\fk_3,\k)^{\k^{\s_3}}$ is generated  over $\k$ by $\veps,\,\xi_0^0,\,\xi_1^1,\,\xi_2^2$ and $\xi_3$. In particular if $\eta\in\Z^2(\fk_3,\k)^{\k^{\s_3}}$, then it is of the form
	\begin{align*}
		\eta=\kappa\veps+\eta_0\xi_0^0+\eta_1\xi_1^1+\eta_2\xi_2^2+\eta_3\xi_3
	\end{align*}
	where
	$\kappa\coloneqq\eta(1,1)$, $\eta_i\coloneqq\eta(x_i,x_i)$, $i\in X$, and $\eta_3\coloneqq\eta(x_0,x_1x_0x_2)$.
\end{proposition}
\pf
Let $\eta\in \Z^2(\fk_3,\k)^{\k^{\s_3}}$ be a $\k^{\s_3}$-linear Hochschild 2-cocycle. 
Recall form \S\ref{subsec:hoch2} that such $\eta$ is determined by the scalar $\eta(1,1)$ and the other 33 scalars $\eta(x_i,b)$, $i\in X$, $b\in\mathbb{B}^+$.
Next, we discard the values $\eta(x_i,b)$ which are necessarily zero while simultaneously we show that the generators in the statement are $\k^{\s_3}$-linear.
For each $i\in X$ and $b\in \BB$, we analyze the equations
\begin{align*}
	\delta_t\cdot \eta(x_i,b)=\delta_{e,t}\eta(x_i, b)=\sum_{p\in\s_3}\eta(\delta_p\cdot x_i, \delta_{p^{-1}t}\cdot b)=\eta(x_i, \delta_{g_i^{-1}t}\cdot b).
\end{align*}

If $b=x_j$, this gives $\eta(x_i, x_j)=\delta_{i,j}$; which shows that $\xi^i_i\in \Z^2(\fk_3,\k)^{\k^{\s_3}}$, $i\in X$.
When $b\in\BB$ has degree two, say $b=x_jx_k$, $j,k\in X$, this equation reads $\delta_{e,t}\eta(x_i,b)=\delta_{g_i^{-1}t,g_jg_k}\eta(x_i, b)$. Thus $\eta(x_i,b)=0$ as 
$g_jg_k\neq g_i$ for any $i,j,k\in X$. 

Next, it readily follows that $\eta(x_i,b)=0$ when $b=x_ix_jx_i$, $j\neq i\in X$ or  $b=x_0x_1x_0x_2$, by the left identity in \eqref{eqn:conds-hoch} and relation $x_i^2=0$. 

We end the proof by checking that $\xi_{3}\in \Z^2(\fk_3,\k)^{\k^{\s_3}}$, as the equation becomes: 
\begin{align*}
	\delta_{e,t}\eta(x_i, x_jx_kx_j)=\delta_{g_i^{-1}t,g_jg_kg_j}\eta(x_i, x_jx_kx_j)
\end{align*}
which is tautological, since $g_i=g_jg_kg_j$ and $g_i^2=e$, for  different $i,j,k\in X$.
\epf

\subsection{Exponentials}
In this part we compute all cocycles $ \eta \in \Z^2(\fk_3,\k)^{\k^{\s_3}}$ such that $e^{\eta}$ is multiplicative.
Recall that we can assume $\eta(1,1)=0$. Thus, there are scalars $(\eta_i)_{0\leq i\leq 3}\in\k$ such that 
\begin{align}\label{eq:eta-fk3-COdiagonal}
	\eta=\eta_0\xi_0^0+\eta_1\xi_1^1+\eta_2\xi_2^2+\eta_{3}\xi_3.
\end{align}

First, $e^{\eta}\in Z^2(\fk_3)$ when $\eta$ verifies \eqref{eqn:conm1}.
We rapidly characterize such cocycles.

\begin{lemma} If 
	$\eta\in \Z^2(\fk_3,\k)^{\k^{\s_3}}$ satisfies  \eqref{eqn:conm1}, then
	\[
	\eta\in\C\coloneqq\{\eta_{i}\xi_i^i,\, \eta_3\xi_3:\eta_i,\eta_3\in\k,\,i\in X\}.
	\]
\end{lemma}
Once again, the converse follows, but the case $\eta=\xi_3$ can be checked only after a long and tedious computation. The case $\eta=\eta_i\xi_i^i$  follows using \cite[Lemma 2.3]{GaM}.
\pf
Let $\eta$ be such that it verifies \eqref{eqn:conm1}. Now, let us fix $x\in\fk_3$, $y=x_ix_j$ and $z=x_k$, with $i,j,k\in X$ all different. Then first identity becomes
\[
\eta(x,x_j)\eta(x_k,x_k)=0
\]
and thus $\eta_j\eta_k=0$ and $\eta_3\eta_k=0$, $i\neq k\in X$.
Therefore, $\eta\in \C$.
\epf

We can now fully determine the Hopf cocycles that arise as exponentials.
We shall write $\sigma_{(0,\eta_i,0)}$ for the cocycle $\sigma_{(c_0,c_1,c_2)}$ with a single nonzero index, given by the value $\eta_i$, in position $i\in\{1,2,3\}$.

Once again we get a set $\overline{\C}$ with $\C\subsetneq\overline{\C}$: see Remark \ref{rem:exp-copointed} for analysis. 

\begin{proposition}\label{pro:COfk3-converse}
	Fix $\eta\in \Z^2(\fk_3,\k)^{\k^{\s_3}}$. Then	$e^\eta\in Z^2(\fk_3)$  if and only if 
	\begin{align*}
		\eta\in \overline{\C}\coloneqq\{\eta_i\xi_i^i+\eta_3\xi_3 :\eta_i,\eta_3\in\k,\,i\in X\}.
	\end{align*}
\end{proposition}

As before, we extract a technical lemma from the proof, for clarity.
\begin{lemma}\label{lem:tecnical-c}
	Let $\eta\in\bar{\C}$. If $a,b\in\mathbb{B}^+$, then 
	\begin{align*}
		e^{\eta}(a,b)=\begin{cases}
			\eta(a,b), & \text{if }\deg(a)+\deg(b)< 6\\
			\frac{1}{2}\eta^{\ast 2}(a,b), & \text{if }\deg(a)+\deg(b)\geq  6.
		\end{cases}
	\end{align*}
	Moreover, if $\eta=\eta_i\xi_i^i+\eta_3\xi_3$, $i\in X$, then  it is given by the table:
	\begin{center}
		\begin{table}[H]
			\resizebox{12cm}{!}
			{\begin{tabular}{|c|c|c|c|c|c|c|c|c|c|c|c|c|c|}
					\hline
					$e^{\eta}$  & $x_0$ & $x_1$ & $x_2$ & $x_0x_1$ & $x_0x_2$& $x_1x_0$& $x_1x_2$ & $x_0x_1x_0$ & $x_0x_1x_2$ & $x_1x_0x_2$ & $x_0x_1x_0x_2$\\
					\hline
					$x_0$  & $\delta_{i,0}\eta_0$ &  &  &  &  &  &  &  &  & $\eta_3$ & \\
					\hline
					$x_1$  &  & $\delta_{i,1}\eta_1$ &  &  &  & &  &  & $\eta_3$ &  & \\
					\hline
					$x_2$  &  &  & $\delta_{i,2}\eta_2$ &  &  & &  & $-\eta_3$ &  &  & \\
					\hline
					$x_0x_1$  &  &  &  &  & $\eta_3$ &  &  &  &  &  & \\
					\hline
					$x_0x_2$  &  &  &  & $\eta_3$ &  & & $-\eta_3$ &  &  &  & \\
					\hline
					$x_1x_0$  &  &  &  &  &  & & $\eta_3$ &  &  &  & \\
					\hline
					$x_1x_2$  &  &  &  &  & $-\eta_3$ & $\eta_3$ &  &  &  &  & \\
					\hline
					$x_0x_1x_0$  &  &  & $\eta_3$ &  &  & &  &  &  &  & \\
					\hline
					$x_0x_1x_2$  &  & $-\eta_3$ &  &  &  & &  &  &  &  & \\
					\hline
					$x_1x_0x_2$  & $-\eta_3$ &  &  &  &  & &  &  &  &  & \\
					\hline
					$x_0x_1x_0x_2$  &  &  &  &  &  & &  &  &  &  & $-\eta_3^2$\\
					\hline
			\end{tabular}}
		\end{table}
	\end{center}
\end{lemma}	
\pf
The first paragraphs in the proof of Lemma \ref{lem:tecnical-p} apply and we just need to deal with $e^{\eta}(a,b)$, for $a,b\in \mathbb{B}$ with $\deg a + \deg b\in\{6,8\}$; as in {\it loc.cit.}~we have that $e^{\eta}(a,b)=\eta(a,b)$ otherwise.

Now, let $a=x_lx_mx_l$, $b=x_ix_jx_i\in \fk_3$, with $i, j, k\in X$ and $l,m,n\in X$ all different respectively. As in the pointed case, we show $\eta^{\ast3}(a,b)=0$. In this case, we observe that in the development of $\Delta^2(a\ot b)$ is impossible obtain a pure tensor of the form $x_t^{\ot 6}$, $t\in X$.

When $a=b=x_0x_1x_0x_2$, we get $\eta^{\ast4}(a,b)=0$  by a similar argument. On the  other hand, notice that in the development of $\eta^{\ast3}(a,b)$ each term has a factor of the form $\eta_i\eta_j$ or $\eta(x_i,x_j)$, $i\neq j\in X$, which is zero. Therefore for  $\eta^{\ast3}(a,b)=0$.

We use \gap\, to calculate $\eta^{\ast 2}$ and end the proof, see \texttt{eta2-copointed.g}.
\epf

We can now show Proposition \ref{pro:COfk3-converse}.
\pf
Let $\eta\in \Z^2(\fk_3,\k)^{\k^{\s_3}}$ be such that $e^\eta$ is a Hopf 2-cocycle, $\eta$ as in \eqref{eq:eta-fk3-COdiagonal}.
By definition of the exponential, we get that $e^\eta(x_jx_i,x_ix_j)=\eta_i\eta_j+\eta_j\eta_k$. On the other hand, the decomposition formula \eqref{eqn:dec-form2} leads to $e^\eta(x_jx_i,x_ix_j)=\eta_i\eta_j$. Hence, $\eta_j\eta_k=0$, for $j\neq k\in X$; that is $\eta\in \overline{\C}$.

Conversely, let $\eta\in\overline{\C}$; say $\eta=\eta_i\xi_i^i+\eta_3\xi_3$, $i\in X$. 
We start by pointing out that the case $\eta_3=0$ is clear, as we get $\sigma_{(0,\eta_i,0)}=e^{\eta_i\xi_i^i}$, by comparing tables. 
Next we look at the orbit $\alpha_d\ra\sigma_{(0,\eta_i,0)}$ as in Theorem \ref{cocicloCOdiagonal}: we have that $\alpha_{-\eta_3}\ra\sigma_{(0,\eta_i,0)}$ coincides with  $e^{\eta}$, hence $e^\eta\in Z^2(\fk_3)$.
\epf

\begin{remark}\label{rem:exp-copointed}
	Although the inclusion $\C\subset \overline{\C}$ is strict, these sets coincide up to coboundaries, as $\overline{\C}=\C+\k\{\xi_3\}$ in $\Z^2(\fk_3,\k)$.
\end{remark}

Next theorem characterizes the cocycles $\sigma$ which are pure.

\begin{theorem}\label{thm:CO-fk3-sigma-pure}
	Fix $\cc=(c_0,c_1,c_2)$ and let $\sigma_{\cc}$ be a 2-cocycle as in Theorem \ref{cocicloCOdiagonal}. Then the following are equivalent:
	\begin{enumerate}[label=(\alph*)]
		\item $\sigma_{\cc}\sim e^\eta$ for some $\eta\in \Z^2(\fk_3,\k)^{\k^{\s_3}}$.
		\item There is a unique no null $c_i\in\{c_0,c_1,c_2\}$.
	\end{enumerate}
\end{theorem}
\pf
{\it (a)} $\Rightarrow${\it (b)}. Let $d\in\k$ and $\eta\in \Z^2(\fk_3,\k)^{\k^{\s_3}}$ be such that $\sigma_{\cc}=\alpha_d\ra e^\eta$. Hence $e^\eta=\alpha_{-d}\ra \sigma_{\cc}$ is a Hopf cocycle and consequently $\eta=\eta_{i}\xi_i^i+\eta_3\xi_3\in\overline{\C}$, for some $i\in X$. By Theorem \ref{cocicloCOdiagonal}, if follows that $c_j=0$ if $j\neq i$ and $c_i=\eta_i$. As well, we get that $d=\eta_3$.

{\it (b)} $\Rightarrow${\it (a)}. By looking at the tables in Theorem \ref{cocicloCOdiagonal} and the proof of Proposition \ref{pro:COfk3-converse}, we see that in this case $\sigma_{\cc´~}=e^{c_i\xi_i^i}$.
\epf

Next remark contributes to the motivation for Question \ref{q}, {\it cf.}~Remark \ref{rem:question-p}.
\begin{remark}\label{rem:question-c}
	We have that $\ex(\C)=\ex(\overline{\C})$ in $H^2(\fk_3)$. Indeed, let us fix $e^\eta\in \ex(\overline{\C})\subset Z^2(\fk_3)$. Hence $\eta=\eta_i\xi_i^i+\eta_3\xi_3 \in \overline{\C}$, for some fixed $i\in\{0,1,2\}$ and  $\eta_i,\eta_3\in \k$. 
	It follows that $e^\eta=\alpha_{-\eta_3} \ra \sigma_{(0, \eta_i, 0)}=\alpha_{-\eta_3} \ra e^\xi$, and $\xi=  \eta_i\xi_i^i \in \C$.
\end{remark}

Now we look at these cocycles in the purely Hopf algebra setting. Recall the definition of the copointed Hopf algebras $A_{\bs\alpha}$ in \eqref{eqn:relations-copointed}, which classify copointed Hopf algebras over $\k^{\s_3}$. They all arise as cocycle deformations of the graded Hopf algebra $\fk_3\#\k^{\s_3}$. Next corollary shows that the cocycles  are generically pure.

\begin{corollary}\label{cor:copointed-exponential}
	$A_{\bs\alpha}$ is isomorphic to a nontrivial exponential cocycle deformation of $A=\fk_3\#\k^{\s_3}$ only if $[\bs\alpha]=[1:0]$.
\end{corollary}
\pf
Let $\sigma=e^\eta\#\eps$ be a Hopf cocycle given by the exponential of a Hochschild 2-cocycle $\eta$ on $\fk_3$. 
In particular, $\eta \in \overline{\C}$ by Proposition \ref{pro:COfk3-converse} and $e^\eta$ is cohomologous to either the trivial cocycle $\veps$ or one the following: $\sigma_{(c,0,0)}$, $\sigma_{(0,c,0)}$ or $\sigma_{(0,0,c)}$, by Theorem \ref{thm:CO-fk3-sigma-pure}.
As an example, if $e^\eta\sim\sigma_{(c,0,0)}$, then 
\[
A_\sigma\simeq A_{\sigma_{(c,0,0)\#\eps}}\stackrel{\eqref{eqn:alphas}}{\simeq} A_{(c,c)}\simeq A_{(1,1)}\stackrel{\eqref{eqn:alphas-action}}{\simeq} A_{(-1,0)}\simeq A_{[1:0]}.\]
As well, if $e^\eta\sim\sigma_{(0,c,0)}$ or $e^\eta\sim\sigma_{(0,0,c)}$, then  $A_\sigma\simeq A_{(0,1)}\stackrel{\eqref{eqn:alphas-action}}{\simeq} A_{[1:0]}$.
\epf

\end{document}